\documentclass[10pt]{amsart}
\usepackage{amssymb,amscd,amsmath, amsthm, ifthen}
\usepackage[all]{xy}
  
\usepackage[small,nohug,heads=littlevee]{diagrams}

\diagramstyle[labelstyle=\scriptstyle]

\usepackage{color} %include even if images aren’t in color
\usepackage{pstricks,pdftricks} 

\newboolean{hide}
\setboolean{hide}{true}
\newboolean{modular}
\setboolean{modular}{true}
\newboolean{crossratio}
\setboolean{crossratio}{false}
\newboolean{concise}
\setboolean{concise}{false}
\newboolean{complexified}
\setboolean{complexified}{false}

\newcounter{commento}

\renewcommand{\comment}[1]{
\addtocounter{commento}{1}
\marginpar{\sffamily{\tiny \thecommento #1
\par
}\normalfont}
}
\renewcommand{\comment}[1]{}

\def\ul{\underline}
\def\sign{\mathrm{sign}}
\def\zero{\vec{0}}

\def\phr{\varphi}
\def\BB{\mathcal{B}}
\def\CC{\mathcal{C}}
\def\DD{\mathcal{D}}
\def\MM{\mathcal{M}}
\def\VV{\mathcal{V}}

\def\Kf{\mathbb{K}}

\def\Cm{\mathbf{C}}
\def\Vm{\mathbf{V}}
\def\Bm{\mathbf{B}}
\def\Mm{M}

\def\Sm{\mathbf{S}}
\def\Km{\mathbf{K}}

\def\Com{\mathcal{C}}
\def\Vom{\mathcal{V}}
\def\Mom{\MM}

\def\Co{\mathcal{C}}
\def\Do{\mathcal{D}}
\def\Mo{\mathcal{M}}
\def\Uo{\mathcal{U}}
\def\So{\mathcal{S}}

\def\ph{\operatorname{ph}}
\def\sgn{\operatorname{sign}}

\def\relint{\operatorname{relint}}
\def\conv{\operatorname{conv}}
\def\pconv{\operatorname{pconv}}
\def\min{\operatorname{min}}
\def\supp{\mathrm{supp}}

\def\min{\mathrm{min}}

\def\fun{\phr_{ \CC}}
\def\cal{\mathcal}

\def\cl{\operatorname{cl}}

\newtheorem{thmA}{Theorem}
\newtheorem{thm}{Theorem}[section]
\newtheorem{lemma}[thm]{Lemma}
\newtheorem{prop}[thm]{Proposition}

\newtheorem{cor}[thm]{Corollary}

\theoremstyle{definition}
\newtheorem{defn}[thm]{Definition}

\newtheorem{remark}[thm]{Remark}
\newtheorem{example}[thm]{Example}

\def\beginwval#1#2{\bgroup% change to \the#1 should be local
  \edef\@savecount{\the\value{#1}}% save old value of counter
  \expandafter\def\csname the#1\endcsname{#2}% new value
  \begin{#1}}% begin environment
\def\endwval#1{\setcounter{#1}{\@savecount}\end{#1}\egroup}

\def\rank{\mathop{\rm rank}}
\def\row{\mathop{\rm row}}
\def\ker{\mathop{\rm ker}}

\begin{document}

%\begin{frontmatter}
\title{Foundations for a theory of complex matroids}
%\\ {\it \normalsize P\MakeLowercase{hirotopes, circuits, duality, and a sad but inevitable absence of vectors.}}}
%\end{center}
\author{Laura Anderson}
%\ead{laura@math.binghamton.edu}
\address{Department of Mathematical Sciences, Binghamton University, Binghamton, NY 13902-6000, USA.}
\author{Emanuele Delucchi}
%\ead{delucchi@math.uni-bremen.de}
\address{Department of Mathematics, University of Bremen, Bibliothekstrasse 1, 28359 Bremen, Germany.}

\maketitle

\begin{abstract}
We explore a combinatorial theory of linear dependency in complex space, {\em complex matroids}, with foundations analogous to those for oriented matroids. We give multiple equivalent axiomatizations of complex matroids,  showing that this theory captures properties of linear dependency, orthogonality, and determinants over $\mathbb C$ in much the same way that oriented matroids capture the same properties over $\mathbb R$. In addition, our complex matroids come with a canonical $S^1$ action analogous to the action of $\mathbb C^*$ on a complex vector space.

Our {\em phirotopes} (analogues of determinants) are the same as those studied previously by Below, Krummeck, and Richter-Gebert \cite{BKR} and Delucchi \cite{Del}. 

We further show that complex matroids cannot have vector axioms analogous to those for oriented matroids. 
\end{abstract}
% \begin{keyword}
% Matroids, Oriented matroids, Complex matroids, Hyperplane arrangements, Grassmann-Pl\"ucker relations, Circuit elimination, Modular pairs, Grassmannians.
% \end{keyword}

%\end{frontmatter}

\section*{Introduction}

Our motivation in this paper lies at the intersection of topology, geometry, and combinatorics. Matroids and oriented matroids have proved to be a valuable tool for studying geometric and topological objects defined in terms of vector spaces. More precisely, matroids constitute a relatively crude tool for studying objects defined over arbitrary vector spaces, while oriented matroids offer a more refined theory for the study of objects defined over $\mathbb R^n$. We offer two examples:
\begin{itemize}
\item To every finite set of hyperplanes in a vector space over a field $\mathbb K$ there is an associated matroid defined by the pattern of intersections of the hyperplanes. In the case $\mathbb K = \mathbb C$, the complement of this set of hyperplanes in the vector space has interesting topology. The cohomology ring of this space depends only on the associated matroid, see \cite{Yu}, but  Rybnikov gave examples of families of hyperplanes having the same matroid but having complements with nonisomorphic fundamental groups \cite{ABC}. However, if the defining equations of the hyperplanes have real coefficients, the data encoded in the corresponding oriented matroid determines the homotopy type of the complement \cite{Sal}.
\item There is a canonical function from the set of rank $r$ subspaces of a vector space $\mathbb K^n$ to the set of rank $r$ matroids on elements $[n]$. When $\mathbb K$ is a topological space, this leads to a topological map from the Grassmannian $G(r,\mathbb K^n)$ to the order complex of the poset of all rank $r$ matroids on $[n]$. This map is topologically uninteresting, because this order complex is contractible. However, when $\mathbb K=\mathbb R$ we get a map from $G(k,\mathbb R^n)$ to the poset of all rank $r$ {\em oriented} matroids on $[n]$, and this map preserves considerable topology (\cite{AD},\cite{Ahomotopy}).
\end{itemize}

Our aim here is a theory of complex matroids that will play a similar role for complex objects to what oriented matroids play for real objects. The hope is for useful tools for studying objects such as complex hyperplane arrangements and complex vector bundles. An early positive sign in this sense is given by the fact that our construction does distinguish the two examples by Rybnikov. The remainder of this introduction will lay out the perspective on matroids and oriented matroids that we wish to generalize.
\vskip 12pt

This paper focuses on axiomatics for complex matroids. To begin with, we briefly  review matroids and oriented matroids. 

A matroid on a finite set $E$ can be given by any of the following:
\begin{enumerate}
\item its set $B(M)$ of {\em bases}
\item its set $V(M)$ of {\em vectors} -- this terminology is not much used in matroid theory, but a {\em vector} of a matroid is just the complement of a {\em flat}
\item its set $C(M)$ of {\em circuits}
\item the "orthogonal complements", in the appropriate sense, of each of the above: {\em cobases $B^*(M)$, covectors $V^*(M)$,} and {\em cocircuits} $C^*(M)$.
\end{enumerate}

The terminology is that these sets are {\em cryptomorphic} -- each of these sets determines the other sets. There are axiomatic definitions of each of these sets -- so any one of these definitions can be taken as the definition of a matroid.

For example, a matrix $\cal M$ over a field, with columns $v_{e_1}, \ldots, v_{e_n}$ indexed by $E$, gives a matroid $M$ with
\begin{enumerate}
\item $B(M)=\{A\subseteq E: \{v_a: a\in A\}\mbox{ is a basis for the column space of $\cal M$}\}$ 
\item $V(M)=\{\supp(x):x\in\ker(\cal M)\}$
\item $C(M)$ is the set of minimal nonzero elements of $V(M)$
\item $B^*(M)=\{E-S: S\in B(M)$
\item $V^*(M)=\{\supp(x):x\in\row(\cal M)\}$
\item $C^*(M)$ is the set of minimal nonzero elements of $V^*(M)$
\end{enumerate}
In this case $\cal M$ is called a {\em realization} of $M$.

Moving beyond matroids, one might consider some specific field $\Kf$ and look for additional structure to put on matroids to reflect properties special to matrices over $\Kf$. 

In the case $\Kf =\mathbb R$ (or any ordered field) this search has been wildly successful: the result is {\em oriented matroids}, introduced by Folkman and Lawrence \cite{FL}. 
Oriented matroids are matroids with extra structure. Broadly put, each data set
described above for matroids realized by a matrix $M$ over a field $\Kf$ says whether various elements of $\Kf$ are zero or nonzero, while the corresponding data set for oriented matroids realized over $\mathbb R$ describes whether these elements of $\mathbb R$ are zero, positive, or negative.
 As a shorthand for this we shall say that the {\em structure set} for matroids is $\{0,\, \neq\! 0\}$, while the structure set for oriented matroids is $\{0, +, -\}$. Thus oriented matroids have cryptomorphic axiom systems:
\begin{itemize} 
\item {\em signed basis axioms}, better known as {\em chirotope axioms}, which in the case of a matroid arising from a matrix $M$ over $\mathbb R$ describe the signs of all nonzero maximal minors of $M$;
\item {\em signed vector axioms}, which in the case of a matroid arising from a matrix $M$ over $\mathbb R$ describe $\{\sign(x): x\in\ker(M)\}$;
\item {\em signed circuit axioms}, which in the case of a matroid arising from a matrix $M$ over $\mathbb R$ describe the elements of $\{\sign(x): x\in\ker(M)\}$ of minimal nonempty support (where $\sign(x_1, \ldots, x_n)=(\sign(x_1), \ldots, \sign(x_n))$).
\end{itemize}

Further, oriented matroids have a notion of duality that is compatible with duality of ordinary matroids and reflects orthogonality of subspaces of $\mathbb R^n$. If $\Mom$ is an oriented matroid with set of signed vectors $\Vom$, then the set $\Vom^*$ of signed vectors of the dual $\Mom^*$ is called the set of {\em signed covectors} of $\Mom$, and the set $\Com^*$ of signed circuits of $\Mom^*$ is called the set of {\em signed cocircuits} of $\Mom^*$

Perhaps the most wonderful property of oriented matroids is the Topological Representation Theorem (\cite{FL}). This theorem says that the nonzero covectors of a rank $d$ oriented matroid correspond to the cells in a very intuitive cell decomposition of $S^{d-1}$. In fact, yet another crytomorphic defintion of oriented matroids can be given in terms of these cell decompositions.

Now consider the case $\Kf=\mathbb C$: what is the right notion of ``complex matroid''? In contrast to oriented matroids, the development here has been limited. Ideally, one would hope for cryptomorphic axiom systems similar to those for oriented matroids, resulting in a Topological Representation Theorem.

Ziegler~\cite{Zie} defined a notion of {complex matroid} with extra structure given by the structure set 
$\{0,+,-,i,-i\}$.  That is, where the set of covectors of a matroid realized by a matrix $M$  over a field $\Kf$ says whether various elements of $\Kf$ are zero or nonzero, and the corresponding data set for oriented matroids realized over $\mathbb R$ describes whether these elements of $\mathbb R$ are zero, positive, or negative,  the corresponding data set for Ziegler's complex  matroids realized over $\mathbb C$ describes whether these elements of $\mathbb C$ are zero, positive real, negative real, have positive imaginary part, or have negative imaginary part. Ziegler's complex matroids have a Topological Representation Theorem \cite[Theorem 3.5]{Zie}. However, they are only known to have one axiomatization, in terms of covectors \cite[Definition 1.3 and 4.1]{Zie}. Ziegler's theory is completely discrete, which can be seen as either a strength or a weakness -- his complex matroids lack any symmetry analogous to the action of $\mathbb C^*$ on complex linear objects.

 Below, Krummeck, and Richter-Gebert~\cite{BKR} developed another notion of complex matroid, with structure set  $S^1\cup\{0\}$, where $S^1$ is the set of unit elements in $\mathbb C$, and with axiomatization only in terms of bases with structure, or {\em phirotopes}. That is, where the set of bases of a matroid realized by a matrix $M$  over a field $\Kf$ says whether various maximal minors of $M$ are zero or nonzero,  the corresponding data set for phirotopes realized by a matrix over $\mathbb C$ additionally describes 
the phase $\theta$ of each nonzero maximal minor $re^{i\theta}$.  Below, Krummeck, and Richter-Gebert gave an axiomatization for phirotopes and proved various interesting properties in rank 2, in particular about realizability. Delucchi  \cite{Del}
 developed a notion of orthogonality for this context, leading to dual phirotopes, and defined circuits and cocircuits associated to a phirotope (although he did not find circuit axioms).

Taking the point of view of the theory of {\em matroids with coefficients} developed by Dress and Wenzel \cite{DW}, phirotopes correspond to basis orientations over the fuzzy ring $\mathbb C // \mathbb R^+$, of which $S^1\cup\{0\}$ is a subset. Within this framework, Dress and Wenzel show phirotopes to be cryptomorphic to what can be roughly taken to be an axiomatization for ``signed flats'' (with coefficients in the full fuzzy ring), and one can prove that dual pairs of matroids with coefficients have ``orthogonal'' signatures. However, Dress and Wenzel's work gives no cryptomorphic axiomatization of matroids with coefficients in terms of dual pairs, nor in terms of circuits.

In the present paper we ask (and, to some extent, answer) how much of the foundations of oriented matroids can be paralleled with the structure set $S^1\cup\{0\}$. We give two different axiomatizations for circuits and cocircuits of a complex matroid and show them to be cryptomorphic to the phirotope axioms. We then give two examples that draw distinctions between oriented matroids and complex matroids: first, that the circuit axioms for complex matroids must have a more restricted form than those for oriented matroids, and that the there is no ``good'' set of vector or covector axioms. 
Finally, we briefly discuss weak maps of complex matroids.

\subsection*{Acknowledgements} We thank Tom Zaslavsky, with whom we discussed early versions of the work, and the referee, whose comments greatly improved the presentation. The second author would like to thank Eva-Maria Feichtner for advising him during his diploma thesis, in which some of the topics of this work were addressed. The first author would like to thank Eva-Maria Feichtner for introducing her to the second author.

\section{Complex matroids}\label{sect:COM}

This section outlines our main results and should serve the reader as a road map through the remainder of the paper. We start by defining complex phases and putting some notation in place. Then, we present our cryptomorphic axiomatizations for complex matroids. We close by sketching the discussion about covectors, complexification and weak maps that will take place in the last sections of the paper.

\subsection{Complex phases}

\begin{defn}[Phase vectors] Given a finite ground set $E$, a {\em phase vector} (or {\em ``phased set''}) is any  
$$X\in (S^1\cup\{0\})^E$$
where $S^1=\{z\in\mathbb{C}\mid \vert z\vert = 1\}$ is the unit circle in the Gauss plane of the complex numbers. 
We will denote by $X(e)$ the $e$-th component of $X$.
We define a partial order on phases by setting $0< \mu$ for all $\mu\in S^1$ and declaring any two elements of $S^1$ to be incomparable. This extends to a partial order on phase vectors, defined componentwise. The minimal phase vector with respect to this ordering is the zero vector, which has value $0$ on every component and will be denoted by $\zero$.

 The {\em phase} $\ph(x)$ of $x\in{\mathbb C}$ is defined to be 0 if $x=0$ and $\frac{x}{|x|}$ otherwise. For $v\in{\mathbb C}^E$, $\ph(v)$ is defined to be the vector with components $(\ph(v))_e=\ph(v_e)$. 
\end{defn}

\begin{defn} Define the {\em phase convex hull} $\pconv(S)$ of a finite $S\subset S^1\cup\{0\}$ to be the set of all phases of (real) positive linear combinations of $S$. Thus
\begin{itemize}
\item $\pconv(\emptyset)=\emptyset$,
\item $\pconv(\{\mu\})=\{\mu\}$ for all $\mu$,
\item $\pconv(\{\mu, -\mu\})=\{0, \mu,-\mu\}$ for all $\mu$,
\item if $S=\{e^{i\alpha_1}, \ldots, e^{i\alpha_k}\}$ with $k\geq 2$ and $\alpha_1<\cdots<\alpha_k<\alpha_1+\pi$, then $$\pconv(S)=\pconv(S\cup\{0\})=\{e^{i\gamma}\mid \alpha_1<\gamma<\alpha_k\}$$
\item if $S=\{e^{i\alpha_1}, \ldots, e^{i\alpha_k}\}$ with $k\geq 3$ and $\alpha_1<\cdots<\alpha_k=\alpha_1+\pi$, then $$\pconv(S)=\pconv(S\cup\{0\})=\{e^{i\gamma}\mid\alpha_1<\gamma<\alpha_k\},$$
\item otherwise (i.e., if the nonzero elements of $S$ do not lie in a closed half-circle of $S^1$) $\pconv(S)= S^1\cup\{0\}$.
\end{itemize}
\end{defn}

\subsection{Axioms for complex matroids}

\begin{defn}[Phirotope axioms, compare \cite{BKR}]\label{def:COM-phir}$\,$

A function $\phr:E^d\rightarrow S^1\cup\{0\}$  is called a {\em rank $d$ phirotope} if
\begin{itemize}
\item [($\phr\,$1)] $\phr$ is nonzero
\item[($\phr\,$2)] $\phr$ is alternating
\item[($\phr\,$3)] For any two subsets $x_1,\ldots,x_{d+1}$ and $y_1,\ldots , y_{d-1}$ of $E$, $$0\in\pconv(\{(-1)^k\phr(x_1,x_2,\ldots,\hat{x_k} ,\ldots, x_{d+1})\phr(x_k,y_1,\ldots,y_{d-1})\}).$$
\end{itemize}
\end{defn}

\begin{defn}[Phased circuit axioms]\label{def:COM-cir}\label{def:COM:circ}$\,$
A set $ \Co\subseteq (S^1\cup\{0\})^E$ is the {\em set of phased circuits of a complex matroid} if and only if it satisfies
\begin{enumerate}
\item[($\Co 0$)] for all $X\in \Co$ and all $\alpha\in S^1$, $\alpha X\in \Co$ {\em (Symmetry)}
\item[($\Co 1$)] for all $X,Y\in \Co$ with $\supp(X)=\supp(Y)$, $X=\alpha Y$ for some $\alpha\in S^1$ {\em (Incomparability)}
%\item[(SE)]  for all $X,Y\in  \Co$, $e,f\in E$ with $X(e)=-Y(e)\neq 0 $ and $Y(f)\neq X(f)$, there is $Z\in  \Co$ with $f\in \supp(Z)\subseteq \supp(X)\cup\supp(Y)\setminus e$.
\item[(ME)] for all $X,Y\in \Co$  such that  $\supp(X)$, $\supp(Y)$ is a modular pair in $\{\supp(X)\mid X\in \Co\}$ and all $e,f\in E$ with $X(e)=-Y(e)\neq 0$ and $X(f)\neq -Y(f)$, there is $Z\in \Co$ with 
\begin{itemize}
\item [$\bullet$]  $f\in \supp(Z)\subseteq (\supp(X)\cup\supp(Y))\setminus e $, and 
\item [$\bullet$]  \(
\left\{\begin{array}{ll}
Z(g)\in\pconv(\{ X(g),Y(g)\})  &\textrm{if } g\in\supp(X)\cap\supp(Y)\\
Z(g)\leq\max\{X(g),Y(g)\} & \textrm{else}\\ 
\end{array}\right.\)\\ 
\end{itemize}{\em (Modular Elimination).}

\end{enumerate}

\end{defn}

\begin{remark}\label{obs:COM} Some points about matroids:
\begin{itemize}
\item[(1)] The phirotope axioms imply that the support of $\phr$ is the set of bases of a matroid $\Mm_\phr$. 
\item[(2)] See Definition \ref{modgen} for a definition of ``modular pair''.  By Lemma \ref{lemma:modelim}, property ($\Co 0$), ($\Co 1$) and (ME) together  
show that the set $\{\supp(X) \mid X\in  \Co\}$ is the set of circuits of a matroid $\Mm_\Co$. 
\end{itemize}
\end{remark}

\begin{remark} Some points about realizability:
\begin{itemize}
\item[(1)] It is easily seen that if $M$ is a rank $d$ matrix over
  $\mathbb C$ with columns indexed by $E$ then the function $E^d \to
  S^1\cup\{0\}$ taking each $d$-tuple to the phase of the determinant
  of the corresponding submatrix of $M$ is a phirotope. In this case
  Property ($\phr\,$3) follows from the Grassmann-Pl\"ucker relations.\\
  Similarly, the set $\Co$ of all phase vectors $\ph(v)$, where $v$ runs over
  all elements of $\ker(M)$ of minimal nonzero support, is the set of phased
  circuits of a complex matroid. 
  We call $M$ a {\em realization} of $\phr$ resp. $\Co$.

\item[(2)] Theorem \ref{thm:elim} will give a correspondence between phirotopes and sets of phased circuits of complex
  matroids. A corresponding $\phr$ and $\Co$ will be called ``the
  phirotope  resp.\ the set of phased circuits of a {\em complex matroid}.'' A phirotope $\phr$ has realization $M$ if and only if its corresponding $\Co$ has realization $M$. In this case we will call $M$ a realization of the complex matroid. A complex matroid that admits a realization is
  called {\em realizable}.

%A phirotope and a set of phased circuits of a complex
%  matroid admitting a common realization $M$ correspond through the
%  bijection of Theorem \ref{thm:elim}. Thus we can say that $M$ is a
%  realization of the complex matroid defined by the given phirotope
%  resp. set of phased circuits.  A complex matroid that admits a realization is
%  called {\em realizable}.
  
  \item[(3)] We say that a subspace $W$ of $\mathbb C^E$ is a realization of a given complex matroid if $W=\ker(M)$ for some matrix realization $M$ of the complex matroid.
\end{itemize}
\end{remark}

\begin{remark}\label{obs:modular} Some points about modularity:
\begin{itemize}
\item[(1)] The form of Elimination in our phased circuit axioms is perhaps the most surprising element here. As Example  \ref{ex:el:bad} will show,  the set of phased
  circuits of a complex matroid need {\em not} satisfy a general 
  Elimination Axiom analogous to that for oriented matroids (Axiom
  $\Com 2$ in Definition \ref{OMaxioms}), even in the realizable
  case. Based on this
  example, our feeling is that any general Elimination Axiom that
  is weak enough to hold for all complex matroids will not be strong
  enough to support a notion of duality (i.e., to prove Proposition
  \ref{circ_perp}).
  
\item[(2)]  Our Modular Elimination Axiom is reminiscent of a characterization of oriented matroids due to 
  Las Vergnas (Theorem 2.1 of \cite{LasV}).  However, Las Vergnas's "Modular Elimination" describes a criterion to check whether a signature of an {\em already given} set of circuits
  of a matroid satisfies an elimination condition, whereas we do not assume
  that an underlying matroid is given.
As discussed in Appendix \ref{appendix:OMmod} , in the oriented
  matroid context either our Modular Elimination %, Las Vergnas Modular Elimination,
  or general Elimination can be taken as an axiom. 
   
   For complex matroids, our choice of Modular Elimination Axiom, without the assumption of an underlying matroid,  allows sleeker proofs and presents complex
  matroids not as `built on top' of matroids but as `matroids on
  a different structure set'.

  \end{itemize}
 
\end{remark}

\begin{defn}\label{def:circorient} If $\Mm$ is a matroid and $\Co$ is the set of phased circuits of a complex matroid such that $\Mm_\Co=\Mm$, we say $\Co$ is a {\em complex circuit orientation} of $\Mm$.
\end{defn}

\begin{defn}\label{phr:completion} For a rank $d$ phirotope $\phr$ on
  the ground set  $E$, we
  say that a subset $\{e_1, \ldots, e_k\}\subseteq E$ is {\em $\phr$-independent} if
  it is an independent set of the matroid $M_\phr$.  
  We call a maximal $\phr$-independent set a  {\em $\phr$-basis}.
\end{defn}

\begin{defn} We say two phirotopes $\phr_1$, $\phr_2$ are {\em equivalent} if $\phr_1=\alpha\phr_2$ for some $\alpha\in S^1$.
\end{defn}

\begin{thmA}\label{thm:elim}
There is a bijection between the set of all equivalence classes of phirotopes on a set $E$ and the set of all sets of phased circuits of complex matroids on $E$, determined as follows. For a phirotope $\phr$ and the corresponding set $\Co$ of phased circuits,
\begin{itemize}
\item[(1)] The set of all supports of elements of $\Co$ is the set of minimal nonempty $\phr$-dependent sets, and
\item[(2)] The phases of $X\in\Co$ are determined by the rule 
\[
\frac{X(x_i)}{X(x_0)}=(-1)^{i}\frac{\phr(x_0,\ldots,\widehat{x_i},\ldots,x_d)}{\phr(x_1,\ldots,x_d)}
\] for all $i=0,\ldots,k$, 
where $x_0\in \supp(X)$ and $\{x_1,\ldots, x_d\}$ is any $\phr$-basis containing $\supp(X)\setminus x_0$.

\end{itemize}
\end{thmA}

%{\bf Proof of Theorem~\ref{thm: elim}:
  \begin{proof}
    Definition~\ref{defn:complexpivotrule} associates to each equivalence class of phirotopes
    $\phr$ a set of phased sets satisfying the two conditions listed
    in the theorem, and Section~\ref{phiro>circs} shows that this
    collection satisfies the circuit axioms.

    The converse follows in two steps:
    \begin{itemize}
    \item In Sections \ref{phiro>pairs} and \ref{pairs>phiro} we get a bijection between
      equivalence classes of phirotopes and {\em dual pairs of circuit
        signatures}, defined in Definition~\ref{def:COM:dualpair}.
    \item Section~\ref{circs>pairs} derives, for each set $\Co$ of
      phased circuits of a complex matroid, a set $\Do$ of phased sets
      so that $\Co, \Do$ is a dual pair of circuit signatures.
    \end{itemize}

The structure of the proof is summarized in the chart depicted in
Figure \ref{fig_chart}.
%\scalebox{0.8}
{
\begin{figure}[h]
\begin{center}
\scalebox{0.85}{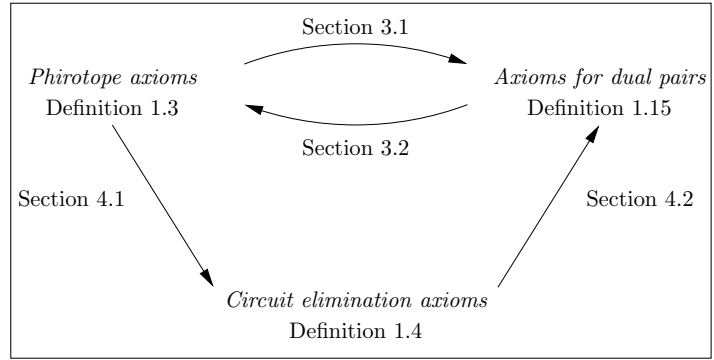}
\end{center}
\caption{Structure of the Proof of Theorem \ref{thm:elim}.}\label{fig_chart}
\end{figure}
}

  \end{proof}

Thus we can refer to ``the complex matroid with phirotope $\phr$ and phased circuit set $\Co$''. 
%We will also refer to a matrix over $\mathbb C$ ``realizing a complex matroid'' (not just realizing a phirotope).

\begin{cor} With the notation introduced in Remark \ref{obs:COM}, if $\Mo$ is a complex matroid with phirotope $\phr$ and phased circuit set $\Co$, then $\Mm_\phr=\Mm_\Co$.
\end{cor}
We call this matroid the {\em underlying matroid} of $\Mo$. The rank of $\Mo$ is the rank of its underlying matroid.

 Consider two vectors $v,w\in\mathbb{C}^E$. By definition, they are orthogonal if their (Hermitian) scalar product equals zero: $\langle v,w\rangle=\sum v_e\overline{w_e}=0$. Now, $\ph(v_e\overline{w_e})=\ph(v_e)\ph(w_e)^{-1}$, and if complex numbers with such phases must add up to zero, then the point $0$ in the complex plane must be contained in \begin{center}$\pconv(\{\ph(v_e)\ph(w_e)^{-1}\mid e\in E\})$. 
\end{center}This suggests the following definition.

\begin {defn}[Orthogonality]\label{def:orth:COM}
Let $S,T\in (S^1\cup\{0\})^E$ be two phased sets for some finite set $E$. Let
\[
P_{S,T}=\bigg\{\frac{S(e)}{T(e)}\,\bigg\vert\, e\in\supp(S)\cap\supp(T)\bigg\}.
\]
We say $S$ and $T$ 
are {\em orthogonal}, written $S\perp T$, if 
\[0\in\pconv (P_{S,T}).\]

Two sets $\mathcal{S},\mathcal{T}\subseteq (S^1\cup\{0\})^E$ are called orthogonal, written $\mathcal{S}\perp \mathcal{T}$, if $S\perp T$ for all $S\in\mathcal{S}$ and all $T\in\mathcal{T}$. The set of all phased sets orthogonal to $\mathcal{S}$ is denoted $\mathcal{S}^\perp$.
\end {defn}

\begin{figure}[h]
  \centering
  \scalebox{0.85}{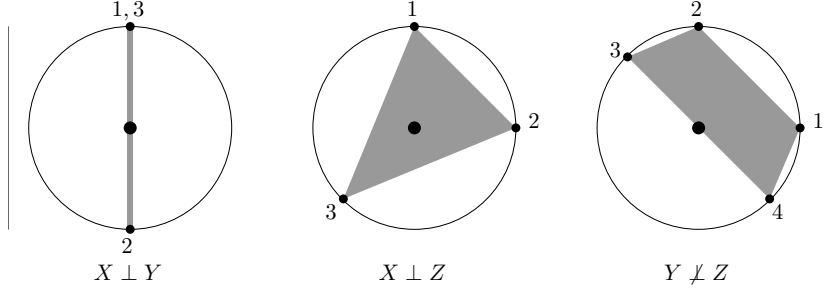}
  \caption{The orthogonality relations between the phased vectors
    $X=(i,1,-1,0)$, $Y=(1,i,i,1)$, $Z=(1,1,e^{-i\pi/4}, e^{i\pi/4})$,
    considered as phased sets with $E=\{1,2,3,4\}$. In the
    picture the index $i$ denotes the position of $X_i/Y_i$
    (respectively, $X_i/Y_i$ and $Y_i/Z_i$).}
\end{figure}

The notion of orthogonality introduced above behaves naturally with respect to duality.

\begin{thmA}\label{dualCOM} If $\Mo$ is a complex matroid with ordered ground set $E$, phirotope $\phr: E^d\to S^1\cup\{0\}$, and circuit set $\Co$, then there is a complex matroid $\Mo^*$ with ground set $E$ and
\begin{itemize}
\item[(1)] phirotope $\phr^*: E^{|E|-d}\to S^1\cup\{0\}$ given as in
  Definition \ref{def:phr_dual},
\item[(2)] circuit set $\Co^*=\min(\Co^\perp\setminus \{\underline{0}\})$,\\ where $\min$ denotes support inclusion minimality.
\end{itemize}

The underlying matroid of $\Mo^*$ is the dual of the underlying matroid of $\Mo$. If $\Mo$ is realized by a vector space $W\subset \mathbb C^E$ then $\Mo^*$ is realized by $W^\perp$.
\end{thmA}

%{\bf Proof of Theorem B:

\begin{proof}
  Lemma~\ref{lem:phirdual} proves that the function $\phr^*$ given in
  Definition \ref{def:phr_dual} is a phirotope with the correct
  underlying matroid. Definition~\ref{defn:complexpivotrule}
  associates to $\phr$ resp. $\phr^*$ sets of phased sets $\Co_\phr$,
  $\Do_{\phr}$, and Proposition~\ref{pairsfromphirs} shows that
  $\Co_\phr$, $\Do_{\phr}$ is a dual pair of complex circuit
  signatures. As proved in Theorem A, $\Co_\phr$ is just $\Co$ and
  $\Do_\phr$ is the circuit set of the complex matroid associated to
  $\phr^*$.  Proposition~\ref{prop:circ:min} shows that
  $\Do_{\phr}=\min(\Co^\perp\setminus
  \{\underline{0}\})$.
  
  If $\Mo$ is realized by a vector space $W$, then $\Co$ is the set of elements of $\{\ph(w)\mid w\in W\backslash\{\vec{0}\}\}$ of minimal support. Certainly the set $\Do$ of minimal elements of $\{\ph(w)\mid w\in W^\perp\backslash\{\vec{0}\}\}$ is contained in $\Co^\perp$. Further, $\Do$ is a complex circuit signature of the dual of the underlying matroid of $\Mo$. Thus by Propositions~\ref{circ_perp} and~\ref{prop:circ:min}, $\Do=\Co^*$.

\end{proof}

\begin{defn} We call the set $\Co^*$ of Theorem~\ref{dualCOM} the {\em set of phased cocircuits} of $\Mo$.
\end{defn}
\comment{check that this is used consistently throughout.}

\begin{remark} The reader will perhaps notice a ``missing item'' in the statement of Theorem \ref{dualCOM} as compared to its counterpart for oriented matroids, Theorem \ref{dualOM}. We will show in Section \ref{Sect:vectors} that there can be no axiomatic description of the phases of the row space of a matrix with complex coefficients (i.e., a "vector axiomatization") that is cryptomorphic to the other axiomatizations.
\end{remark}

The following gives a complex matroid version of a lesser-known characterization of oriented matroids (\cite{BLV78}). 

\begin{defn}\label{def:COM:dualpair} Let $M$ be a matroid with ground set $E$. We say $\Co\subset (S^1\cup\{0\})^E$ is a {\em complex circuit signature} of $M$ if
\begin{enumerate}
\item[(S1)] for all $X\in \Co$ and all $\alpha\in S^1$, $\alpha X\in \Co$, 
\item[(S2)] for all $X,Y\in \Co$ with $\supp(X)=\supp(Y)$, $X=\alpha Y$ for some $\alpha\in S^1$, and
\item[(S3)]  the set $\{\supp(X) \mid X\in  \Co\}$ is the set of circuits of $M$ 
\end{enumerate}
We say $\Do\subset (S^1\cup\{0\})^E$ is a {\em complex cocircuit signature} of $M$ if $\Do$ is a complex circuit signature of $M^*$.

We say $\Co, \Do$ are a {\em dual pair of complex circuit signatures}
of $M$ if  $\Co$ is a complex circuit signature of $M$, $\Do$  is a
complex cocircuit signature of $M$, and 
\begin{enumerate}
  \item[(S4)]$\Co\perp\Do$.
\end{enumerate}

\end{defn}

\begin{thmA}\label{thm:dualpair:COM}
Let $\Co$  be a complex circuit signature and $\Do$ be a complex cocircuit signature of a matroid $M$. Then $\Co$ and $\Do$ are the set of phased circuits and cocircuits of a complex matroid with underlying matroid $M$ if and only if  \[\Co\perp \Do.\]
\end{thmA}

\begin{proof}
  If $\Co\perp \Do$ then $\Co,\Do$ is a dual pair, and the construction of a phirotope with corresponding circuit set $\Co$ from this pair is
  carried out in Section \ref{pairs>phiro}. Conversely, given the set
  $\Co$ of phased circuits and the set  $\Do$ of phased cocircuits of
  a complex matroid with underlying matroid $M$, we know by Theorem~\ref{thm:elim} that there is a phirotope
  $\phr$ with $\Co=\Co_\phr$ and $\Do=\Co_{\phr^*}$. Proposition
  \ref{pairsfromphirs} proves that these are a dual pair of complex
  circuit signatures of $M$.
\end{proof}

\subsection{Minors of complex matroids}

\begin{defn}\label{def:minor:phsets}
For $X\in (S^1\cup\{0\})^E$  and $A\subseteq E$ let $X_{\setminus A}\in (S^1\cup\{0\})^{E\setminus A}$ be the restriction of $X$ to $E\setminus A$. For $\Uo\subseteq (S^1\cup\{0\})^E$ define
\begin{itemize}
\item[(1)] the {\em deletion} of $A$ from $\Uo$ as 
\[
\Uo\setminus A = \{ X_{\setminus A} \mid X\in \Uo, \, \supp(X)\cap A = \emptyset \}.
\]
\item[(2)] the {\em contraction} of $A$ in $\Uo$ as 
\[ 
\Uo /A := \min\{ X_{\setminus A} \mid X\in \Uo \},
\]
where $\min$ denotes support minimality.
\end{itemize}
\end{defn}
%
%\begin{defn} A {\emph basis} for a complex matroid is a basis of the underlying matroid.
%with phirotope $\phr$ is a set $\{b_1, \ldots, b_d

\begin{thmA}\label{thm:minors}
Let $\Co$ be the set of phased circuits of a complex matroid $\Mo$ on the ground set $E$ with underlying matroid $\Mm$. If $A\subseteq E$, then $\Co\setminus A$ is the set of phased circuits of a complex matroid $\Mo\setminus A$ with underlying matroid $\Mm\setminus A$, and $\Co/A$ is the set of phased circuits of a complex matroid $\Mo/A$ with underlying matroid $\Mm/A$. 

Further,
%\begin{itemize}
%\item[(1)] 
with the notation of Definition~\ref{def:minor:phsets} and Theorem~\ref{dualCOM}, 
\[\Co^*/A=(\Co\setminus A)^*.\]

%\item[(2)] Let $\phr$ be a phirotope of $\Mo$. Let $\{a_1, \ldots, a_c\}$ be a basis for $\Mm\setminus A$ and let $\{a_1, \ldots, a_c, b_1, \ldots, b_{d-c}\}$ be a basis for $\Mm$. (The set $\{b_1, \ldots, b_{d-c}\}$ will in some cases be empty.)
%Then the function
%$$\begin{array}{rcl}
%\phr{\setminus A}: E^c&\to& S^1\cup\{0\}\\
%\phr{\setminus A}(x_1,\ldots, x_c)&=&\phr(x_1,\ldots, x_c, b_1, \ldots, b_{d-c})
%\end{array}$$
%is a phirotope for $\Mo\setminus A$, and the function
%$$\begin{array}{rcl}
%\phr{/ A}: E^{d-c}&\to& S^1\cup\{0\}\\
%\phr{/ A}(x_1,\ldots, x_{d-c})&=&\phr(x_1,\ldots, x_{d-c}, a_1, \ldots, a_{c})
%\end{array}$$
%is a phirotope for $\Mo/ A$.
%After replacing real signs with complex phases, Definition
%  [...] gives the phirotopes associated to $\Co\setminus A$
%  and $\Co/A$ in terms of the phirotope associated to $\Co$.
%\comment{Really? do we want this? Which definition is meant here?}
%\end{itemize}
 \end{thmA}

 \begin{remark}
   The phirotopes $\phr\setminus A$ and $\phr/A$ of $\Mo\setminus A$
   resp. $\Mo/A$ are given in Lemmas \ref{lem:phir_contr} and
   \ref{lem:phir_del}.
 \end{remark}
 \begin{proof}
   Lemma \ref{lem:phir_contr} and \ref{lem:phir_del} prove that, given
   a phirotope $\phr$ with underlying matroid $M$, the
   functions $\phr\setminus A$ and $\phr/ A$ are indeed phirotopes with
   underlying matroids $M\setminus A$ resp. $M/ A$. 
   % Proposition \ref{minor:CM} proves that $\Co/A$ and $\Co\setminus A$
%    are sets of phased circuits of two complex matroids with the
%    correct underlying matroid. 
   Proposition \ref{pairsfromphirs} proves
   that
   $\Co/A=\Co_{\phr/A}$ and $\Co\setminus A = \Co_{\phr\setminus
     A}$. The last part of Lemma \ref{lem:phir_del}, together with Theorem~\ref{thm:elim},  then proves the duality result.
   
 \end{proof}

The complex matroids 
%associated to $\Co\setminus A$ and $\Co/A$ are denoted 
$\Mo\setminus A$ and $\Mo/A$ are called respectively the {\em deletion} of $A$ from $\Mo$ and the {\em contraction} of $A$ in $\Mo$. 

\section{Phirotopes, duality and minors}

This section deals with phirotopes as defined in Definition~\ref{def:COM-phir}. Its goal is to establish some basic facts about duality and minors in terms of phirotopes. 

\subsection{Duality} \label{sec:phiro:dual}

Recall from Section~\ref{sect:COM} that given a phirotope $\phr$ on the ground set $E$, the set $\Bm_{\phr}:=\{\{b_1,\ldots,b_d\} \mid \phr(b_1,\ldots , b_d)\neq 0 \}$ is the set of bases of the underlying matroid $\Mm_\phr$.

\begin{defn}\label{def:phr_dual}
Given a  rank d phirotope $\phr$, choose a total ordering of $E$, and for all $(x_1,x_2,\ldots, x_{n-d})\in E^{n-d}$ let $(x'_1,\ldots x'_d)$ be a permutation of $E\setminus \{x_1,\ldots ,x_{n-d}\}$. Define the {\em dual} of $\phr$ as $$\phr^*(x_1,\ldots , x_{n-d}):=\phr(x'_1,\ldots,x'_d)^{-1}\sgn(x_1,\ldots,x_{n-d},x'_1,\ldots,x'_d).$$
\end{defn}

Notice that, up to a global change of sign, $\phr^*$ is independent of the choice of orderings on $E$ and $\{x'_1,\ldots x'_d\}$.

\begin{lemma}\label{lem:phirdual}   $\phr^*$ is a  rank $(n-d)$ phirotope, and the underlying matroid $M_{\phr^*}$ is the dual $(M_\phr)^*$ to $M_\phr$.
\end{lemma}
\begin{proof} By definition, $\Bm_{\phr^*}=\{E\setminus B \mid B\in \Bm_{\phr}\}$ which, by Theorem~\ref{dual:mat}, is the set of bases of $(\Mm_\phr)^*$.  Thus, to prove the lemma it suffices to prove that $\phr^*$ is indeed a phirotope.

Axioms ($\phr\,$1) and ($\phr\,$2) are clear from the definition. For ($\phr\,$3), consider two sets $X:=\{x_0,\ldots ,x_{n-d}\}$ and $Y:=\{y_1,\ldots ,y_{n-d-1}\}$, numbered such that $X\cap Y=\{x_{n-d-l},\ldots x_{n-d}\}=\{y_1,\ldots , y_l\}$. Without loss of generality we can assume that the total ordering of $E$ is given by
$$
x_0,\ldots ,x_{n-d},y_{l+1},\ldots,y_{n-d-1}, A,
$$
where $A$ is any total ordering of $E\setminus (X\cap Y)$.

Then we have
$$
\phr^*(x_0\ldots,\hat{x_k},\ldots,x_{n-d})\phr^*(x_k,y_1,\ldots,y_{n-d-1})=
$$
$$
\phr(x_k,y_{l+1},\ldots y_{n-d-1},A)^{-1}
\underbrace{\sgn(x_0\ldots,\hat{x_k},\ldots,x_{n-d},x_k,y_{l+1},\ldots y_{n-d-1},A)}_{\sigma_1}
$$

$$
\phr(x_0,\ldots,\hat{x_k},\ldots x_{n-d-l}, A)^{-1}
\underbrace{\sgn(x_k,y_1,\ldots,y_{n-d-1},x_0,\ldots,\hat{x_k},\ldots x_{n-d-l}, A  )}_{\sigma_2}
$$
where the sign
$$
\sigma_1\sigma_2=
$$
$$
(-1)^{n-d-k}\sgn(x_0,\ldots,x_{n-d},y_{l+1},\ldots,y_{n-d-1},A)
$$
$$
(-1)^{n-d+k}\sgn(y_1,\ldots , y_{n-d-1},x_0,\ldots,x_{n-d-l},A)
$$
$$
=\sgn(x_0,\ldots,x_{n-d},y_{l+1},\ldots,y_{n-d-1},A)\sgn(y_1,\ldots , y_{n-d-1},x_0,\ldots,x_{n-d-l},A)
$$
does not depend on $k$. 
Then, 
$$
\{(-1)^k\phr^*(x_0\ldots,\hat{x_k},\ldots,x_{n-d})\phr^*(x_k,y_1,\ldots,y_{n-d-1})\mid x_k\in X\setminus Y \}=
$$
$$
\sigma_1\sigma_2\{(-1)^k\phr(x_k,y_{l+1},\ldots y_{n-d-1},A)^{-1}\phr(x_0,\ldots,\hat{x_k},\ldots x_{n-d-l}, A)^{-1}\mid x_k\in X\setminus Y \}.
$$
We now have to prove that $0$ is in the relative interior of the convex hull of the latter set. Equivalently, we want to show that there are positive real numbers $\lambda_k$ such that
\begin{equation}\label{primasomma}
\sum_{k}\lambda_k(-1)^k\phr(x_k,y_{l+1},\ldots y_{n-d-1},A)^{-1}\phr(x_0,\ldots,\hat{x_k},\ldots x_{n-d-l}, A)^{-1} =0.
\end{equation}

Because $\phr$ is a phirotope, we know that there are positive real numbers $\lambda_k$ with
\begin{equation}\label{secondasomma}
\sum_{k}\lambda_k(-1)^k\phr(x_k,y_{l+1},\ldots y_{n-d-1},A)\phr(x_0,\ldots,\hat{x_k},\ldots x_{n-d-l}, A) =0.
\end{equation}
Since Equation (\ref{primasomma}) is the complex conjugate of Equation (\ref{secondasomma}), the claim follows.
\end{proof}

\subsection{Deletion and contraction} \label{sec:phiro:minors}

The following two lemmas will be part of the proof of Theorem~\ref{thm:minors}, as well as various inductive arguments throughout the paper.

\begin{lemma}\label{lem:phir_contr} Let $A\subset E$ be given, and choose a maximal $\phr$-independent subset $\{a_1,a_2,\ldots , a_l\}$ of $A$. Then $$(\phr / A) (x_1,\ldots , x_{d-l}):=\phr (x_1,\ldots x_{d-l}, a_1\ldots , a_l)$$
is a phirotope, and $M_{\phr/A}=M_{\phr}/A$. Up to global multiplication by a constant $c\in S^1$, $\phr/A$ is independent of the choice of $\{a_1,a_2,\ldots , a_l\}$.
\end{lemma}
\begin{proof} The phirotope axioms for $\phr/A$ are easy to check. That $\Mm_{\phr/A}=\Mm_\phr / A$ follows by Definition~\ref{def:mat:minor:bases}.(1) because 
\[
\Bm_{\phr/A}=\{\{x_1,\ldots,x_{d-l}\}\mid \phr(x_1,\ldots , x_{d-l},a_1,\ldots ,a_l)\neq 0\}\]
\[
=\{B\subseteq E \mid B\cup\{a_1,\ldots ,a_l\}\in\Bm_\phr\}.
\]
\end{proof}

\begin{lemma}\label{lem:phir_del} Let $A\subset E$ be given, and let $r$ be the rank of $E\setminus A$ in $M_\phr$. If $r<d$, choose  $\{ a_1, \ldots,a_{d-r}\}\subseteq A$ such that $(E\setminus A) \cup \{a_1,\ldots, a_{d-r}\}$ spans $ M_\phr$. Define a function $\phr\setminus A : E\setminus A\rightarrow S^1\cup \{0\}$ as follows:
$$ (\phr\setminus A) (x_1,\ldots , x_{r}):=
\left\{\begin{array}{ll}\phr(x_1,\ldots x_r) & \textrm{If }r=d \\
\phr(x_1,\ldots , x_r,a_1,\ldots, a_{d-r}), &\textrm{if }r<d.
\end{array}\right.
$$

Then, up to global multiplication by a nonzero constant, $\phr\setminus A$ is independent of the choice of $a_1,\ldots, a_{d-r}$ and  $(\phr\setminus A)^* = \phr^*  / A$ - in particular, it is a phirotope -- and  $M_{\phr \setminus A} = M_{\phr}\setminus A $.
\end{lemma}

\begin{proof} We prove the case where $A=\{a\}$, and we fix a linear ordering of $E$ where $a$ is the biggest element. 

If $r<d$, then $a$ is in every basis of $M_{\phr}$. Thus 
$$\phr^*(x_1,\ldots,x_t)\neq 0 \textrm{ only if }a\not\in\{x_1,\ldots,x_t\},$$
hence
\begin{align*}
(\phr^*/a)(x_{1},\ldots,x_{t})&=\phr^*(x_{1},\ldots,x_{t})\\
&=\phr(x_{t+1},\ldots , x_{n-1},a)^{-1}\sgn(x_1,\ldots ,x_{n-1},a)\\
&=(\phr\setminus a)(x_{t+1},\ldots , x_{n-1})^{-1}\sgn(x_1,\ldots ,x_{n-1})\\
&=(\phr\setminus a)^*(x_1,\ldots,x_t).
\end{align*}

\noindent If on the other hand $r=d$, then
\begin{align*}
(\phr^*/a)(x_{1},\ldots,x_{t})&=\phr^*(x_{1},\ldots,x_{t},a)\\
&=\phr(x_{t+1},\ldots , x_{n-1})^{-1}\sgn(x_1,\ldots,x_t,a,x_{t+1} ,\ldots, x_{n-1})\\
&=\phr(x_{t+1},\ldots , x_{n-1})^{-1}(-1)^{n-t-2}\sgn(x_1,\ldots,x_{n-1},a)\\
&=(\phr\setminus a)(x_{t+1},\ldots , x_{n-1})^{-1}(-1)^{n-t-2}\sgn(x_1,\ldots ,x_{n-1},a)\\
&=(-1)^{n-t-2}(\phr\setminus a)^*(x_1,\ldots,x_t).
\end{align*}	
\end{proof}

\section{Cryptomorphism between phirotopes and dual pairs}\label{sec:dualpairs}

\subsection{Dual pairs from phirotopes}\label{phiro>pairs}

The point of this section is to prove Proposition \ref{pairsfromphirs}, asserting that every phirotope $\phr$ induces a dual pair of complex circuit and cocircuit signatures on $\Mm_{\phr}$.

\begin{lemma} Let $\phr$ be a phirotope and $\Mm_\phr$ its underlying matroid. Let $C$ be a circuit of $\Mm_\phr$, $e,f\in C$, and $\{f, x_2,\ldots,x_d\}$ a basis for $\Mm_\phr$ containing $C\setminus e$. Then the number 
\[
\frac{\phr(e,x_2,\ldots,x_d)}{\phr(f,x_2,\ldots,x_d)}
\]
does not depend on the choice of $x_i$.
\end{lemma}
\begin{proof} Let $\{f, x_{2},\ldots,x_{d-1},x_d'\}$ be another basis for $\Mm_\phr$ containing $C\setminus e$.
 Then axiom ($\phr\,$3) for $\phr$ applied to $\{e,f,x_2,\ldots ,x_d\}$ and $\{x_2,\ldots,x_{d-1},x_d'\}$ reduces to 
$$\phr(f,x_2,\ldots ,x_d)\phr(e,x_2,\ldots ,x_{d-1},x_d')-\phr(e,x_2,\ldots ,x_d)\phr(f,x_2,\ldots ,x_{d-1},x_d')=0$$
and proves the claim for pairs of $\phr$-bases that differ by one element. The full claim follows by induction on the number of elements by which any two choices of basis differ. 
\end{proof}

\begin{defn}\label{defn:complexpivotrule}
Given a phirotope $\phr$, let $\Co_{\phr}$ be the family of all phased sets $X$ such that 
\begin{itemize}
\item $\supp(X)$ is a circuit of $\Mm_\phr$ and 
\item for all $e,f\in X$ and bases $B=\{f,x_2,\ldots,x_d\}$ with $\supp(X)\backslash e\subseteq B$
 we have 
$$
\frac{X(f)}{X(e)}=-\frac{\phr(e,x_2,\ldots,x_d)}{\phr(f,x_2,\ldots,x_d)}.
$$
\end{itemize}

Notice that for any $c\in S^1$ we have $\Co_{c\phr}=\Co_\phr$. Thus, it makes sense to talk about $\Co_{\phr^*}$, $\Co_{\phr \setminus e}$, and $\CC_{\phr /e}$. Let $\Do_{\phr}:=\Co_{\phr^*}$.
\end{defn}

\begin{prop}\label{pairsfromphirs}
For every phirotope $\phr$ the sets $\Co_\phr$ and $\Do_\phr$ satisfy Definition~\ref{def:COM:dualpair} and are thus a dual pair of complex circuit signatures of the matroid $\Mm_\phr$. Moreover, given an element $e$ of the ground set we have
\begin{itemize}
\item[(1)] $\CC_{\phr \setminus e}=\CC_{\phr}\setminus e$
\item[(2)] $\CC_{\phr /e}=\CC_{\phr}/e$
\end{itemize}
\end{prop}
\begin{proof}
All of the properties in the definition of phased circuits and cocircuits
are clear except (S4). 

To see (S4), let $X\in \Co$ and $Y\in \Do$. If $\supp(X)\cap\supp(Y)=\emptyset$, then $X\perp Y$ by definition. Otherwise, let $\supp{(X)}=\{x_1, \ldots, x_k\}$ and $\supp(Y)=\{y_1, \ldots, y_l\}$, with the elements of $\supp(X)\cap\supp(Y)$ written first. Thus, $x_i=y_i$ for all $i$ less than some value $m$.

We can extend $\supp(X)$ to $\{x_1, \ldots, x_{d+1}\}$ so that every $\{x_1, \ldots, \hat x_k,\ldots x_{d+1}\}$ with $x_k\in\supp(X)$ is a basis for $\Mm_\phr$. Similarly, we extend $\supp(Y)$ to $\{y_1, \ldots, y_{n-d+1}\}$ so that every $\{y_1, \ldots, \hat y_k,\ldots y_{n-d+1}\}$ with $y_k\in\supp(Y)$ is a basis for $\Mm_\phr^*$. Let $\{z_1,\ldots,z_{d-1}\}=E\backslash\{y_1, \ldots, y_{n-d+1}\}$.

The Grassmann-Pl\"ucker relations tell us that $0$ is in the phase convex hull of
$$
\{(-1)^k\phr(x_1, \ldots, \hat x_k, \ldots, x_{d+1})\phr(x_k, z_1,\ldots, z_{d-1})\mid k=1,\ldots d+1\}.
$$
Note that one of the factors of $\phr(x_1, \ldots, \hat x_k, \ldots, x_{d+1})\phr(x_k, z_1,\ldots, z_{d-1})$ will be 0 unless $x_k\in\supp(X)\cap\supp(Y)$. Applying the definition of $\phr^*$, we see that the above set can be written
$$\bigg\{\frac{(-1)^k\phr(x_1, \ldots, \hat x_k, \ldots, x_{d+1})\phr^*(y_1,\ldots,\hat y_k,\ldots,y_{n-d+1})^{-1}}{\sgn(x_k, z_1,\ldots, z_{d-1}, y_1,\ldots,\hat y_k,\ldots,y_{n-d+1})}\,\bigg\vert\begin{array}{l} x_k=y_k,\textrm{ both in}\\\supp(X)\cap\supp(Y)\end{array}\!\!\!\!\bigg\}.
$$

Now note that
$$\sgn(x_k, z_1,\ldots, z_{d-1}, y_1,\ldots,\hat y_k,\ldots,y_{n-d+1})$$
$$=(-1)^{d-1+k}\sgn(z_1,\ldots,z_{d-1},y_1,\ldots,y_{n-d+1})$$ and that
if $0$ is in the phase convex hull of a set $A$ of complex numbers then $0$ is in the phase convex hull of $cA$ for any complex number $c$.

So, multiplying all elements of our set by $$(-1)^{d-1}\sgn(z_1,\ldots,y_1,\ldots,y_{n-d+1})\phr(x_2, \ldots, x_{d+1})^{-1}\phr^*(y_2, \ldots, y_{n-d+1}),$$ we see that $0$ is in the phase convex hull of 
$$\bigg\{\frac{X(x_k)Y(x_k)}{X(x_1)Y(y_1)}\bigg\vert\, x_k\in\supp(X)\cap\supp(Y)\bigg\}.$$
Multiplying all elements of this set by $X(x_1)Y(y_1)$, we see that $X\perp Y$.

That  $\CC_{\phr \setminus e}=\CC_{\phr}\setminus e$ and  $\CC_{\phr /e}=\CC_{\phr}/e$ follows immediately from the definition of $\CC$. 
\end{proof}

\begin{cor}\label{DconFi} Given a phirotope $\phr$, consider $X\in\Co_\phr$ and $Y\in\Do_\phr$ such that $\supp(X)=\{x_0,\ldots , x_l\}$, $\supp(Y)=\{y_1,\ldots,y_h\}$. Choose elements $x_{l+1},\ldots ,x_d$ such that $\{x_1,\ldots,x_d\}\in\Bm_\phr$ and elements $z_2,\ldots,z_d$ that span the hyperplane $E\setminus \supp(Y)$ of $\Mm_\phr$. Then,

(1) for every $x_i, x_j\in\supp(X)$, $$
\frac{X(x_i)}{X(x_j)}=(-1)^{i-j}\frac{\phr(x_0,\ldots,\widehat{x_i},\ldots,x_d)}{\phr(x_0,\ldots,\widehat{x_j},\ldots,x_d)},$$

(2) for every $y_i, y_j\in\supp(Y)$, $$
\frac{Y(y_i)}{Y(y_j)}=\frac{\phr(y_j,z_2,\ldots,z_d)}{\phr(y_i,z_2,\ldots,z_d)}.
$$ 
In particular, $\Do_\phr$ can be defined alternatively as the family of all phased sets $Y\subset (S^1\cup\{0\})^E$ satisfying (2).
\end{cor}

\begin{proof}
The claim (1) follows because $\phr$ is alternating, and thus it is enough to keep track of the permutations involved. 

For claim (2), consider $Z\in\Co_\phr$ such that $\supp(Z)$ is the basic circuit of $y_i$ with respect to $\{y_j,z_2,\ldots,z_d\}$. Then, $\supp(Z)\cap\supp(Y)=\{y_i,y_j\}$, and thus since $Z\perp Y$ we must have
$$
\frac{Y(y_i)}{Y(y_j)}=-\frac{Z(y_i)}{Z(y_j)}=
\frac{\phr(y_j,z_2,\ldots,z_d)}{\phr(y_i,z_2,\ldots ,z_d)}.
$$
\end{proof}

\subsection{Phirotopes from dual pairs}\label{pairs>phiro}

This section will prove Proposition~\ref{perp-phir}, that a dual pair of complex circuit and cocircuit signatures induces a unique equivalence class of phirotopes.
 
Recall the notion of {\em basis
  graph} of a matroid (or see Definition~\ref{def:basisgraph} in the
Appendix) and that, if $B$ is a basis of a matroid $\Mm$ on
the ground set $E$ and $x\in E\setminus B$, then there is a unique
circuit $C(B,x)$ contained in $B\cup\{x\}$, called the
{basic circuit of $x$ with respect to $B$} (for this, see 
Lemma~\ref{basic_circ}).

To construct a phirotope from a dual pair $\Co$, $\Do$ of circuit orientations we
will follow the strategy of \cite[Proposition 3.5.2 (2.\
proof)]{BLSWZ},
 which proves a similar result for oriented matroids. 
 The gist of the proof is as follows.
 \begin{itemize}
 \item We arbitrarily choose one ordered basis $(b_1,\ldots, b_d)$ to have $\phr(b_0,\ldots, b_d)=1$. This defines the phirotope on any permutation of this basis.
 \item Given a definition of the phirotope on all permutations of a basis $B_1=\{e, x_2,\ldots,x_d\}$, consider an adjacent basis $B_2=\{f, x_2,\ldots,x_d\}$ in the basis graph.  Let $X\in\Co$ with $\supp(X)=C(B_1,f)$. Then the relation 
 $$
{\phr(f,x_2,\ldots,x_d)}=-\frac{X(e)}{X(f)}\phr(e,x_2,\ldots,x_d)
$$
(from Definition~\ref{defn:complexpivotrule}) determines $\phr(f, x_2,\ldots,x_d)$. 

\item Thus, for each edge $\{B_1, B_2\}$ in the edge graph, we associate the fraction $\frac{X(f)}{X(e)}$ to the direction from $B_1$ to $B_2$. To find the phirotope on permutations of some basis $B$, we find a path from $\{b_1, \ldots, b_d\}$ to $B$ and multiply the appropriate quotients along this path. 

The hard work of the proof is showing that the definition at $B$ is independent of the path chosen.
 \end{itemize}
 
We first need a preliminary lemma that investigates the values of the signatures of the circuits involved in the basis exchanges of ``triangles'' and ``squares'' of basis graphs.
\begin{lemma}\label{k34} Let $\Co, \Do$ be the set of phased circuits resp. cocircuits of a complex matroid with underlying matroid $\Mm$. 

\begin{itemize}
\item[(1)] Given three distinct elements $e,f,g\in E$ with bases $B_e,B_f,B_g$ of $\Mm$ and $A\subset E$ such that $B_e=A\cup e$, $B_f=A\cup f$, $B_g=A\cup g$, and for all $x,y\in\{e,f,g\}$ consider $X_{x,y}\in\Co$
%any complex signature\comment{complex SIGNATURE-PHASING?} of the basic circuit $C_{x,y}$ of $x$ with respect to $A\cup x$ 
with $\supp(X_{x,y})=C(A\cup x, y)$, 
$$
\frac{X_{e,f}(e)}{X_{e,f}(f)}\frac{X_{f,g}(f)}{X_{f,g}(g)}=-\frac{X_{e,g}(e)}{X_{e,g}(g)}.
$$

\item[(2)] Given three distinct elements $e,f,g\in E$ with bases 
$$
B_{e,f}=A\cup\{e,f\},\, B_{f,g}:=A\cup\{f,g\},\, B_{e,g}:=A\cup\{e,g\}
$$ 
of $\Mm$ for some $A\subset E$, choose any $X\in\CC$ with $\supp(X)=C(B_{e,f},g)$.
%be any orientation of the basic circuit $C$ of $g$ with respect to $B_{e,f}$. 
Then, 
$$
\frac{X(g)}{X(e)}\frac{X(e)}{X(f)} = \frac{X(g)}{X(f)}.
$$

\item[(3)] Consider an independent set $A\subset E$ and distinct elements $e,f,g,h\in E$ such that 
$$
B_1:=A\cup \{f,g\},\,B_2:=A\cup \{e,g\},\,B_1':=A\cup \{f,h\},\,B_2':=A\cup \{e,h\}
$$ 
\noindent are bases of $\Mm$, with 
$$
f\in C_1:=C(B_1,e),\, f\in C_1':=C(B_1',e),
$$
$$ 
g\in C_2:=C(B_1, h),\, g\in C_2':=C(B_2,h).
$$
Then for any $X_1, X_2,X_1',X_2'\in\Co$ with $\supp(X_1)=C_1$, $\supp(X_2)=C_2$, $supp(X_1')=C_1'$, $\supp(X_2')=C_2'$,
%Let $B_1$ be a basis of $\Mm$, and $e\not\in B_1$. Let $f$ be an element of the basic circuit of $e$ with respect to $B_1$ and $B_2$ be the basis obtained exchanging $e$ for $f$ in $B_1$. Suppose that, for some other element $h\in B_2\setminus e$, there is $g\in C(B_1,h)\cup C(B_2,h)$. Then, let $B_1':=(B_1\setminus g)\cup h$ and $B_2':=(B_2\setminus g)\cup h$. It follows that $e\in C(B_1',f)$ and
$$\frac{X_1(e)}{X_1(f)}\frac{X_2(h)}{X_2(g)}=\frac{X_1'(e)}{X_1'(f)}\frac{X_2'(h)}{X_2'(g)}.$$
%,\quad \frac{X_2(h)}{X_2(g)}=\frac{X_2'(h)}{X_2'(g)}$$ 
%where $X_1\in\Co$ is any signature of $C(B_1,h)$ and $X_1'\in\Co$ is any signature of $C(B_2,h)$
\end{itemize}
\end{lemma}

The following diagrams illustrate the three cases of the lemma.
\begin{center}
\begin{tabular}{ccccc}
\begin{diagram}
B_e & \rTo^{\frac{X_{e,f}(f)}{X_{e,f}(e)}} & B_f \\
    & \rdTo_{\frac{X_{e,g}(g)}{X_{e,g}(f)}}& \dTo_{\frac{X_{f,g}(g)}{X_{f,g}(f)}} \\
    &             & B_g
\end{diagram} 
&$\quad\quad\quad$&
\begin{diagram}
B_{e,f} & \rTo^{\frac{X(g)}{X(e)}} & B_{f,g} \\
    & \rdTo_{\frac{X(g)}{X(f)}}& \dTo_{\frac{X(e)}{X(f)}} \\
    &             & B_{e,g}
\end{diagram} 
&$\quad\quad\quad$ &
\begin{diagram}
B_2       &\rTo^{\frac{X_2'(h)}{X_2'(g)}}& B_2' \\ 
\uTo^{\frac{X_1(e)}{X_1(f)}} &              &\uTo_{\frac{X_1'(e)}{X_1'(f)}}\\
B_1       &\rTo^{\frac{X_2(h)}{X_2(g)}} & B_1'
\end{diagram} \\
Case (1) && Case (2) && Case (3)
\end{tabular}
\end{center}

\begin{proof}
\noindent (1) For the cocircuit $D:=E\setminus \cl( A)$, we have $D\cap C(A\cup x, y)=\{x,y\}$ for all $x,y\in\{e,f,g\}$. therefore, for any $Y\in \Do$ with $\supp(Y)=D$ we have $Y\perp X_{x,y}$ for all $x,y\in\{e,f,g\}$ and thus 
$$
\frac{X_{e,f}(e)}{X_{e,f}(f)}\frac{X_{f,g}(f)}{X_{f,g}(g)}= 
\bigg( - \frac{Y(e)}{Y(f)}\bigg)\bigg( - \frac{Y(f)}{Y(g)}\bigg)=\frac{Y(e)}{Y(g)}
=-\frac{X_{e,g}(e)}{X_{e,g}(g)}.
$$

\noindent (2) is evident.

\noindent (3) The claim is trivial when $C_1=C_1'$ and $C_2=C_2'$. If this is not the case, then without loss of generality suppose that $g\in C_1$. Then we can use $C_1$ to eliminate $g$ from $B_1$ (or from $B_2$), and we obtain that $B:=A\cup\{e,f\}$ is a basis. Since $g\in C_1$ implies $h\in C_1'$ (for else one could eliminate $e$ and obtain a circuit contained in $B_1$), we can use $C_1'$ to eliminate $h$ from $B_1'$ (or from $B_2'$). Then the basis graph of the matroid contains

$$
\begin{diagram}
B_2       &\rLine(3,0)&   &       & B_2' \\ 
\uLine(0,3) &\rdLine    & T'  &\ruLine &\uLine(0,3)\\
         & T        & B &   T''   &  \\
          &\ruLine    & T''' &  \rdLine    &\\
B_1       &\rLine(3,0)&  &         & B_1'
\end{diagram}
$$

and we can apply part (1) to the ``triangles" $T,T',T'',T'''$ to conclude.
\end{proof}

\begin{prop}\label{perp-phir} If $ \Co$ and $\Do$ are the phased circuits resp.\ phased cocircuits of a complex matroid, then 
$ \Co= \Co_\phr$ and $ \Do= \Do_\phr$ for a phirotope $\phr$.
Moreover, $\phr$ is uniquely determined up to a nonzero constant.  
\end{prop}
\newcommand{\para}[1]{{\newline\vspace{2pt}\noindent\em #1}}
\begin{proof} In this proof we fix a total ordering $>$ of the ground set $E$ of the underlying matroid $\Mm$. We will often identify a subset $A\subseteq E$ with the corresponding sequence ordered by $>$.
\para{1. Labeling the basis graph.} %Let $\Mm$ denote the underlying matroid, and fix a total ordering $>$ of its ground set $E$. %We fix a basis $B$ of $\Mm$ and 
Consider the basis graph $G$ of $\Mm$.
%We choose an arbitrary orientation of every edge of $G$, and will from now regard $G$ as a directed graph.
We define a function $\gamma$ on ordered pairs of adjacent vertices of $G$. Given two bases $B_1$, $B_2$ of $\Mm$ corresponding to a pair of adjacent vertices of $G$ we define
$$
\gamma(B_1,B_2):=(-1)^{i-j}\frac{X(x_i)}{X(x_j)},
$$ 
where $B_1\cup B_2=\{x_0,\ldots ,x_d\}$, $B_1\setminus B_2=\{x_i\}$, $B_2\setminus B_1=\{x_j\}$, the $x_l$ are numbered in increasing order with respect to $>$, and $X\in  \Co$ is any phased circuit with $\supp(X)=C(B_1,x_j)$.
Clearly, $\gamma(B_1,B_2)=\gamma(B_2,B_1)^{-1}$.

Given any closed path $A=B_0,B_1,B_2,\ldots,B_k=A$ in $G$, 
$$
\prod_{i=0}^{k-1}\gamma(B_i,B_{i+1})=1.
$$
To see this note that by Theorem~\ref{theo:maurer} it is enough to check the cases $k=3,4$, which is easy to do using Lemma \ref{k34} and keeping track of the signs.
\para{2. Construction of the phirotope associated with $\CC,\DD$.}
If we fix a ``basepoint'' $B\in V(G)$, Step 1 above tells us there is a well-defined quantity associated to every $B'\in V(G)$ and given by
$$
\overline{\fun}(B'):=\prod_{i=0}^{k-1}\gamma(B_i,B_{i+1})
$$
where $B=B_0,B_1,\ldots , B_k=B'$ is any path from $B$ to $B'$ in $G$, and the empty product equals $1$.

Now we are ready to define a function $\fun : E^d\rightarrow S^1\cup\{0\}$ as follows. Given $x_1 < x_2 <\ldots < x_d\in E$, let
$$
\fun'(x_1,\ldots x_d):=
\left\{\begin{array}{ll}
0 &\textrm{if } \{x_1,\ldots, x_d\} \not\in V(G),\\
\overline{\fun}(\{x_1,\ldots,x_d\}) &\textrm{else.}
\end{array}\right.
$$

\noindent This function can be extended to any ordered $d$-tuple of elements of $E$ by setting 
$$
\fun(x_1,\ldots ,x_d):=\sgn(\sigma)\fun'(x_{\sigma(1)},\ldots,x_{\sigma(d)}),
$$
where $\sigma$ is a permutation such that $x_{\sigma(i)}<x_{\sigma(j)}$ if $i<j$.
For every $X\in \Co$ let $\supp(X)=\{x_0,x_1,\ldots,x_l\}$ be numbered, as usual, in increasing order with respect to $>$. For all $0\leq i,j\leq l$ we can complete $\supp(X)\setminus x_i$ to a basis of $\Mm$ by a set $\{a_1,\ldots a_m\}$. Then $\supp(X)=C(A_i,x_i)$, where
$A_i:=\{x_0,\ldots,\hat{x}_i,\ldots , x_l,a_1,\ldots a_m\}$. We have
\begin{eqnarray}\label{fi_suX}
\frac{X(x_i)}{X(x_j)}=  (-1)^{i-j}\gamma(A_j,A_i)=(-1)^{i-j} \overline{\fun}(A_j)^{-1}\overline{\fun}(A_i)\\
=
(-1)^{i-j}\frac{\fun(x_0,\ldots,\hat{x}_i,\ldots,x_l,a_1,\ldots,a_m)}{\fun(x_0,\ldots,\hat{x}_j,\ldots,x_l,a_1,\ldots,a_m)}\nonumber
\end{eqnarray}

For any pair of adjacent vertices $B_1,B_2\in V(G)$ with $\{f\}=B_2\setminus B_1$, $\{e\}=B_1\setminus B_2$, the basic circuit $C=C(B_1,f)$ of $\Mm$ intersects the basic circuit $D=C^*(E\setminus B_2, e)$ of $\Mm^*$ in the set $\{e,f\}$. Choose $X\in \Co$, $Y\in\Do$ with $\supp(X)=C$, $\supp(Y)=D$. By $ \Co\perp \Do$ we have 
$$
\frac{X(e)}{Y(e)}=-\frac{X(f)}{Y(f)}
$$
\noindent and so
$$
\frac{Y(e)}{Y(f)}=-\frac{X(e)}{X(f)}=\gamma(B_1,B_2).
$$

%Recall the basepoint $B$ of $G$. If we choose $E\setminus B$ as a basepoint for $G^*$, $\gamma^*$ induces a well-defined function $\fun^*:V(G^*)\rightarrow S^1$. 

For every $Y\in  \DD$ and $e,f\in\supp(Y)$, choose a basis $T$ of the hyperplane $H$ of $\Mm$ defined by $H:=E\setminus \supp(Y)$. Then, $T\cup\{e,f\}$ contains a circuit $C$ with $C\cap\supp(Y)=\{e,f\}$. Writing $T_e=T\cup e$, $T_f=T\cup f $ we have, as above,
\begin{eqnarray}\label{fi_suY}
\frac{Y(e)}{Y(f)}=(-1)^{i-j}\gamma(T_e,T_f)=(-1)^{i-j}\overline{\fun}(T_e)^{-1}\overline{\fun}(T_f) \\
=\frac{\fun(f,t_2,\ldots,t_d)}{\fun(e,t_2,\ldots,t_d)}\nonumber
\end{eqnarray}

\noindent where $e$ and $f$ are respectively $i$-th and $j$-th in the $>$-ordering of $T\cup\{e,f\}$, and $t_2,\ldots t_d$ is any total ordering of $T$. In view of Corollary \ref{DconFi}, equations (\ref{fi_suX}) and (\ref{fi_suY}) show that $ \CC= \CC_{\fun}$, $ \DD= \DD_{\fun}$.
\para{3. Verification of the axioms for phirotopes}
The function $\fun$ we constructed so far is an alternating, nonzero function $E^d\rightarrow  S^1\cup 0$. We now prove that $\fun$ satisfies ($\phr\,$3). To this end, consider any two subsets $S:=\{x_0,\ldots x_d\}\subset E$, $T:=\{y_2\ldots y_d\}\subset E$.  If for some $j$ the set $S\setminus x_j$ is a basis of the underlying matroid $\Mm$, then $S\setminus x_i$ is a basis of $\Mm$ only if $x_i$ is in the basic circuit $C_S$ of $x_j$ with respect to $S\setminus x_j$. Also, $T\cup x_j$ is a basis only if $T$ is an independent set and $x_j$ is in the cocircuit $D_T$ given by the complement of the hyperplane spanned by $T$.

We may from now on suppose that $T$ is independent and $S\setminus x_j$ is a basis of $\Mm$ for some $j$.  Then, the product
$$\fun(x_0,\ldots,\hat{x}_i,\ldots,x_d)\fun(x_i,y_2,\ldots,y_d)$$
is nonzero if and only if $x_i\in C_S\cap D_T$. 

We thus have to consider the set
$$Q:=\{(-1)^i\fun(x_0,\ldots,\hat{x}_i,\ldots,x_d)\fun(x_i,y_2,\ldots,y_d)\mid x_i\in C_S\cap D_T \}$$
and show that $0\in\relint\conv Q$.

Let us suppose without loss of generality that $x_0\in C_S\cap D_T$. Take $X\in  \CC$ such that $\supp(X)=C_S$ and $X(x_0)=1$,  $Y\in \DD$ such that $\supp(Y)=D_T$ and $Y(x_0)=1$.

Then we may consider the rotated set $\mu Q$ for $\mu=\fun(x_1,\ldots,x_d)^{-1}\fun(x_0,y_2,\ldots,y_d)^{-1}$. By equations (\ref{fi_suX}) and (\ref{fi_suY})
\begin{align*}
\mu Q &= 
\bigg\{
(-1)^{i}\frac{\fun(x_0,\ldots,\hat{x}_i,\ldots,x_d)}{\fun(x_1,\ldots,x_d)}
\frac{\fun(x_i,y_2,\ldots,y_d)}{\fun(x_0,y_2,\ldots,y_d)} 
\,\bigg\vert\, x_i\in C_S\cap D_T \bigg\}\\
&=
\bigg\{
-\frac{\fun(x_0,x_1\ldots,\hat{x}_i,\ldots,x_d)}{\fun(x_i,x_1,\ldots,\hat{x}_{i},\ldots,x_d)}
\frac{\fun(x_i,y_2,\ldots,y_d)}{\fun(x_0,y_2,\ldots,y_d)} 
\,\bigg\vert\, x_i\in C_S\cap D_T \bigg\}\\
&=
\bigg\{
\frac{X(x_i)Y(x_0)}{X(x_0)Y(x_i)}
\,\bigg\vert\, x_i\in C_S\cap D_T \bigg\}
=
\bigg\{
\frac{X(x_i)}{Y(x_i)}
\,\bigg\vert\, x_i\in C_S\cap D_T \bigg\},\\
\end{align*}

\noindent thus $0\in\relint \conv Q$ if and only if $0\in(\relint \conv \mu Q) = P_{X,Y}$ - but the latter is the case because, by assumption, $X\perp Y$. 
\end{proof}

\section{From phirotopes to circuits to dual pairs}\label{sec:elimination}

This section will prove the two implications forming the "bottom of the triangle" in Figure~\ref{fig_chart}.

In the following we will often argue by induction on the size of the ground set of the complex matroid. As preparation, we prove that our notion of complex circuit orientation (Definition~\ref{def:circorient}) behaves well with respect to the operations of deletion and contraction as introduced in Definition~\ref{def:minor:phsets}.

\begin{prop}\label{minor:CM}
If $ \Co$ is a complex circuit orientation of a matroid $\Mm$ on $E$, then for all $e\in E$ \begin{itemize}
\item[(1)] $ \Co\setminus e$ is a complex circuit orientation of the matroid $\Mm\setminus e$, and
\item[(2)] $ \Co/e$ is a complex circuit orientation of the matroid $\Mm/e$.
\end{itemize} 
\end{prop}

\begin{proof}
Let $ \Co$ be as in the statement. For (1) note that the elements of 
$ \Co\setminus e$ are all phased circuits in $ \Co$ not containing $e$ in their support, and so 
(ME) holds in $ \Co\setminus e$ because, by Lemma \ref{trans_mod}.(1), a modular pair of circuits in $\Mm\setminus e$ is modular in $\Mm$ too, and the result of modular elimination between them in $\Mm$ is again an element of $\Mm\setminus e$.

For (2), recall first that for every element of  $X\in \Co/e$ there is a unique  element $X'\in \Co$ with $\supp(X)\subseteq \supp(X')\subseteq\supp(X)\cup e$ so that $X'(x)=X(x)$ for all $x\in\supp(X)$.
Lemma \ref{trans_mod}.(2) ensures that for every modular pair $X,Y$ in $ \Co/e$ the corresponding $X',Y'\in \Co$ defined as above also define a modular pair. As above, the element $Z'$ obtained by modular elimination of $f$ between $X'$ and $Y'$ restricts to $Z\in \Co/e$ with $f\in\supp(Z)\subset\supp(X)\cup\supp(Y)$. By the uniqueness of modular elimination we are done. 
\end{proof}

\subsection{From phirotopes to circuit orientations}\label{phiro>circs}

In this section we prove that the set $\Co_\phr$ of circuits induced by a phirotope $\phr$ satisfies the conditions of Definition~\ref{def:COM:circ} for phased circuits. Conditions $(\Co 0)$ and $(\Co 1)$ are clear; we have to prove that (ME) holds in $\Co_\phr$, and as a stepping stone we prove the following ``special elimination'' property.

\begin{lemma}[SE]\label{SE:phr} Let $\phr$ be a phirotope on the ground set $E$.
For all $X,Y\in \Co_\phr$ and $e,f\in\supp(X)\cap\supp(Y)$ such that $X(e)=Y(e)$ and $X(f)\neq Y(f)$, there is $Z\in \Co$ with $f\in\supp(Z)\subseteq \supp(X) \cup \supp(Y)$.
\end{lemma}

\begin{proof}  Suppose by way of contradiction that there are $X,Y\in\Co_\phr$, $e,f\in E$ so that the claim does not hold and let $A:=\supp(X)\setminus\{e,f\}$, $B:=\supp(Y)\setminus\{e,f\}$. Then $f\not \in\cl(A\cup B)$ and we can extend $A$ to $A'$ and $B$ to $B'$, where $A'$ and $B'$ are bases of the hyperplane $H$ containing $\cl(A\cup B)$ but not $e$ (and thus not $f$ either). Then let $D:=E\setminus H$. It follows that $D\cap \supp(X)=D\cap\supp(Y)=\{e,f\}$. If we fix a total ordering of the ground set $E$ we can think of any subset of $E$ as representing an ordered tuple of elements. With $D':=D\setminus\{e,f\}$ we can write 

$$
\frac{X(f)}{X(e)}=
-\frac{\phr(e,A')}{\phr(f,A')} = \frac{\phr^*(e,D', H\setminus A')}{\phr^*(f, D', H\setminus A')}.
  $$
But since this value does not depend on how we complete the set $D'$ to a complement of a basis of $\Mm_{\phr}$, we have

$$\frac{X(f)}{X(e)}= \frac{\phr^*(e,D', H\setminus B')}{\phr^*(f, D', H\setminus B')}
=-\frac{\phr(e,B')}{\phr(f,B')}=\frac{Y(f)}{Y(e)},
$$
contradicting the assumption.
  
\end{proof}

\begin{prop}[ME]\label{phir_circ} Let % $ \Co_\phr$ be the set of phased circuits of a complex matroid given by the phirotope 
$\phr$ be a phirotope. 
For all $X,Y\in \Co_\phr$ with $X\neq \mu Y$ for all $\mu\in S^1$ and such that $\supp(X)$, $\supp(Y)$ is a modular pair of circuits of $\Mm_\phr$, given $e,f\in E$ with $X(e)=-Y(e)\neq 0$ and $X(f)\neq Y(f)$, there is $Z\in \Co_\phr$ with $f\in \supp(Z)\subseteq \supp(X)\cup\supp(Y)\setminus\{e\}$, and 
$$
\left\{\begin{array}{ll}
Z(f)\in\; \pconv(\{ X(f),Y(f)\}) & \textrm{if } f\in\supp(X)\cap\supp(Y),\\
Z(f)\leq\max\{X(f),Y(f)\} & \textrm{else.}\\ %if } f\in\supp(X)\textrm{ or }f\in\supp(Y) \\
\end{array}\right.
$$
\end{prop}
\begin{proof}
For ease of notation and terminology, let us prove this for the dual matroid -- that is, when $X,Y\in \Co_{\phr^\ast}$ are cocircuits of the complex matroid defined by $\phr$. 

Since the supports of $X$, $Y$ form a modular pair, we have $x,y\in E$ and $A\subset E$ such that $\supp(X)=E\setminus\cl(A\cup \{x\})$, $\supp(Y)=E\setminus\cl( A\cup \{y\})$. Then it follows that $x\in\supp(Y)$ and $y\in\supp(X)$, for otherwise $\supp(X)=\supp(Y)$ and $X=\mu Y$ for some $\mu\in S^1$, which cannot be. From now on we fix a total ordering $a_2,\ldots ,a_d$ of $A$ and, when appropriate, write $A$ for $a_2,\ldots,a_d$.

Let  $D$ be the (unique) cocircuit complementary to the hyperplane $E\setminus \cl( A\cup \{e\})$. By definition, the sign vector defined by $Z(x):=Y(x)$ and
$$
\frac{Z(f)}{Z(x)} :=\frac{\phr(x,e,A)}{\phr(f,e,A)}  \textrm{ for all } f\neq x
$$
is a signature of $D$. We will prove that it satisfies the requirements.

First of all, consider the element $y\in \supp(X)\cap\supp(Z)$. We have

$$
\frac{Z(y)}{Z(x)}:=\frac{\phr(x,e,A)}{\phr(y,e,A)}=\frac{\phr(x,e,A)}{\phr(x,y,A)}\frac{\phr(x,y,A)}{\phr(y,e,A)}=-\frac{X(y)}{X(e)}\frac{Y(e)}{Y(x)}=\frac{X(y)}{Y(x)}
$$

\noindent and therefore, since we set $Z(x)=Y(x)$, we obtain $Z(y)=X(y)$.

Now let us consider an element $f\in\supp(Z)\setminus\supp(X)$. Then $f\not\in \supp(X)$, and since $f\not\in \cl( A) $ (for otherwise $f\not\in\supp(Z)$) we conclude that we can exchange $f$ for $x$ in the base $A\cup\{x\}$ of the hyperplane $E\setminus\supp(X)=\cl( A\cup\{x\}) =\cl(A\cup\{f\}) $. Therefore we can compute
$$
\frac{Z(f)}{Z(x)}\frac{Y(x)}{Y(f)}=\frac{\phr(x,e,A)}{\phr(f,e,A)}\frac{\phr(f,y,A)}{\phr(x,y,A)}=\frac{\phr(e,x,A)}{\phr(y,x,A)}\frac{\phr(y,f,A)}{\phr(e,f,A)}=\frac{X(e)}{X(y)}\frac{X(y)}{X(e)}=1,
$$
hence $Z(f)=Y(f)$. By a similar argument we obtain $Z(f)=X(f)$ for every $f\in\supp(Z)\setminus\supp(Y)$.

As the last case, we consider an element $f\in\supp(Z)\cap\supp(X)\cap\supp(Y)$. Because the set $B:=\{e,y\}\cup A$ is a basis of $\Mm_\phr$ and $f$ is not an element of $\cl( A\cup \{y\})$ nor of $\cl( A\cup\{e\})$, the basic circuit $C(f,B)$ of $f$ with respect to $B$ contains $e,y,f$, and thus $C(f,B)\cap \supp(X)=\{e,y,f\}$. In order to compute $Z(f)$, we apply the axiom (b) for phirotopes to the tuples of elements $y,f,e,A$ and $x,A$ and conclude that $0$ must be in the relative interior of the phase convex hull of
 
$$
\{{\phr(f,e,A)}{\phr(y,x,A)},\, -{\phr(y,e,A)}{\phr(f,x,A)},\,{\phr(y,f,A)}{\phr(e,x,A)}\}.
$$

This condition does not depend upon rotation - i.e., multiplication by an element of $S^1$. Thus, after multiplication by $(\phr(y,e,A)\phr(y,x,A))^{-1}$, equivalently we may say
$$
0\in\pconv\bigg\{
\frac{\phr(f,e,A)}{\phr( y,e,A)},\;
-\frac{\phr(f,x,A)}{\phr(y,x,A)},\;
\frac{\phr(y,f,A)}{\phr(y,x,A)}\underbrace{\frac{\phr(e,x,A)}{\phr(y,e,A)}}_{=-\frac{Z(y)}{Z(x)}}\bigg\}
$$
which, by Corollary \ref{DconFi} and after reflection with respect to
the real axis, is equivalent to %, can be rewritten as
$$
0\in\pconv\bigg\{
\frac{Z(f)}{Z(y)},\; 
-\frac{X(f)}{X(y)},\;
-\frac{Y(f)}{Y(x)}\frac{Y(x)}{X(y)}
\bigg\}.
$$
We already established that $Z(y)=X(y)$, and thus multiplying everything by this number we conclude that 
$$
0\in\pconv\big\{
Z(f),\; -X(f),\; -Y(f) 
\big\}
$$
\noindent or, equivalently, $Z(f)\in\pconv(\{ X(f), Y(f)\})$.

\end{proof}

\subsection{From circuit orientations to dual pairs}\label{circs>pairs}

The goal of this section is to ``close the circle'' and show that the axiomatization in terms of circuit elimination given in Definition~\ref{def:COM:circ} is equivalent to the axiomatization for dual pairs of Definition~\ref{def:COM:dualpair}. We will do so by showing that the set of circuits of a complex matroid induces a (unique) orthogonal complex signature of the cocircuits of the underlying matroid.

\begin{lemma}\label{SE:circ} Let $\Co$ be a circuit orientation of a complex matroid. Then 
\begin{itemize}
\item[(SE)] for all $X,Y\in  \Co$, $e,f\in E$ with $X(e)=-Y(e)\neq 0 $ and $Y(f)\neq X(f)$, there is $Z\in  \Co$ with $f\in \supp(Z)\subseteq \supp(X)\cup\supp(Y)\setminus e$.
\end{itemize}
\end{lemma}
\begin{proof}
By Lemma~\ref{lemma:modelim} the set $\Cm:=\{\supp(X)\mid X\in \Co\}$ is the set of circuits of a matroid $\Mm$.

We argue by induction on the rank of the $\Mm$. The claim is trivial in rank $0$ and $1$, and every pair of circuits is modular in rank $2$. So let $\Co$ be a circuit orientation of a complex matroid of rank $d>2$ and suppose the claim holds for all complex matroids of smaller rank.

By way of contradiction, let $X,Y\in\Co$ and $e,f\in E$ be such that for all $C\in\Cm$ with $C\subseteq\supp(X)\cup\supp(Y)$, $f\not\in C$. The case where $X(f)Y(f)=0$ is covered by matroid elimination (Definition~\ref{def:mat}.($\Cm 2$)). So suppose $e,f\in\supp(X)\cap\supp(Y)$ and choose $a\in \supp(Y)\setminus\supp(X)$. By Proposition~\ref{minor:CM}, $\Co/a$ is again a complex orientation of the circuits of the rank $d-1$ matroid $\Mm/a$. By definition there are $X',Y'\in \Co/a$ with $X'(g)\leq X(g)$, $Y'(g)\leq Y(g)$ for all $g\in E\setminus a$, and with $f\in \supp(X')\cap\supp(Y')$. With the notation of Definition~\ref{def:minor:phsets}, $Y'=Y_{\setminus a}$ and thus $e\in\supp(Y')$. 

Now, if  $e\not \in\supp(X')$ we reach a contradiction by taking $C:=\supp(X')\cup a$. Otherwise $e,f\in\supp(X')\cap\supp(Y')$ so $X'(e)=X(e)=-Y(e)=-Y'(e)$ and $X'(f)=X(f)\neq Y(f)=Y'(f)$. We apply induction hypothesis to the rank-$(d-1)$ complex matroid $\Co/a$ and find $Z'\in\Co/a$ with $f\in\supp(Z')\subseteq \supp(X')\cup\supp(Y')\setminus e$. Then we reach a contradiction by taking $C:=\supp(Z')\cup a \in \Cm$.
 \end{proof}

\begin{lemma}\label{lemma_treel} Let $\Mm$ be a matroid on the ground set $E$. Consider a circuit $C$ and a cocircuit $D$ of $\Mm$ such that $\vert C\cap D \vert \geq 3$. Then there are elements $e,f\in D\cap C$ and a cocircuit $D'$ of $\Mm$ such that 

\noindent (1) $D$ and $D'$ are a modular pair,

\noindent (2) $e \in (D'\cap C)\subseteq (D\cap C)\setminus f$.
\end{lemma}

\begin{proof} Let $D$ and $C$ be as above, and let $r$ be the rank of $\Mm$. Then $C\setminus D$ is an independent set of rank at most $r-2$ and can be completed to a basis $B$ of the hyperplane $H:=E\setminus D$.

For every $e\in C\cap D$, the set $B\cup e$ is a basis of $\Mm$. The basic circuit of $f$ with respect to this basis cannot be contained fully in $(C\setminus D)\cup e$, and thus it contains an element $x\in B\setminus (C\setminus D)$. Let $A:=B\setminus x$. Then we have $H=\cl( A\cup x)$ and we can define $$H':=\cl( A\cup f), \quad\quad D':=E\setminus H'.$$

Clearly, $(D'\cap C)\subseteq (D\cap C)\setminus f$. To prove $e\in D'\cap C$, it is enough to show $e\not\in H'$. But if $e$ were in $H'$, then there would be a circuit contained in the set $A\cup\{e,f\}$, and by the uniqueness of basic circuits, this would be also the basic circuit of $f$ with respect to $B\cup e$ - contradicting the definition of $x$. 
\end{proof}

\begin{prop}\label{circ_perp} For any complex circuit orientation $ \Co$ with underlying matroid $\Mm$ there is a unique complex circuit signature $ \Do$ of $\Mm^*$ such that $ \Do\perp \Co$.
\end{prop}

\begin{proof}
Let $ \Co$ be a complex circuit orientation with underlying matroid $\Mm$. 

\noindent {\em Definition of $ \Do$:} For every cocircuit $D$ of $\Mm$, choose a maximal independent subset $A$ of the hyperplane $D^c$. Then for every $e,f\in D$, there is a unique circuit $C_{D,e,f}$ of $\Mm$ with support contained in $A\cup\{e,f\}$. (Namely, $C_{D,e,f}$ is the basic circuit of $f$ with respect to $A\cup e $.) Choose $X_{D,e,f}\in\Co$ with $\supp(X_{D,e,f})=C_{D,e,f}$.

$$
 \Do:=\bigg\{W\in (S^1\cup\{0\})^E \left\vert\begin{array}{l} 
D:=\supp(W)\in \Cm(\Mm^*),\\
 \forall e,f\in \supp(W),\, \frac{W(e)}{W(f)}=-\frac{X_{D,e,f}(e)}{X_{D,e,f}(f)}
\end{array}\right.\bigg\}
$$

Certainly this $\Do$ is the unique candidate for a complex circuit signature of $\Mm^*$ orthogonal to $\Co$. It remains to see that $\Do$ is, in fact, a well-defined complex circuit signature.

\noindent{\bf Claim 1.} {\em $ \Do$ is well-defined and independent of the choice of the $X_{D,e,f}$.}\\
\noindent{\em Proof.} First we prove independence of the choice of the $X_{D,e,f}$. Given $D\in \Cm^*(\Mm)$ and $e,f\in D$, let $Y$ and $Y'$ be two candidates for $X_{D,e,f}$. Multiplying $Y$ by an element of $S^1$, we may assume $Y(e)=-Y'(e)$. If $Y(e)/Y(f)\neq Y'(e)/Y'(f)$, then by Lemma~\ref{SE:circ} there is $Z\in  \Co$ with $\supp(Z)\cap D=\{f\}$, contradicting Lemma \ref{two_el}.

To conclude that $ \Do$ is well-defined, it is enough to prove that, given $D\in \Cm^*(\Mm)$ and $e,f,g\in D$, 
$$
-\frac{X_{D,f,g}(f)}{X_{D,f,g}(g)}=\bigg(-\frac{X_{D,e,f}(f)}{X_{D,e,f}(e)}\bigg)\bigg(-\frac{X_{D,e,g}(e)}{X_{D,e,g}(g)}\bigg).
$$ 

The circuits $C_{D,e,f}$ and $C_{D,e,g}$ form a modular pair, because their complements both contain the corank $2$ coflat $\cl( E\setminus (A\cup\{f,g\}) )$. Then (modular) elimination of $e$ from $X_{D,e,f}$ and $\frac{- X_{D,e,f}(e)}{X_{D,e,g}(e)}X_{D,e,g}$ gives $Y\in  \Co$ with $f,g\in\supp(Y)$ and 
$
\frac{Y(f)}{Y(g)}=\frac{X_{D,e,f}(f)}{- X_{D,e,f}(e)}\frac{X_{D,e,g}(e)}{X_{D,e,g}(g)}.
$ 
So 
$$
\frac{X_{D,f,g}(f)}{X_{D,f,g}(g)}=\frac{Y(f)}{Y(g)}=-\frac{X_{D,e,f}(f)}{X_{D,e,f}(e)}\frac{X_{D,e,g}(e)}{X_{D,e,g}(g)}
$$
and the claim follows.

\noindent{\bf Claim 2.} {\em Fix $W\in  \Do$. For all $X\in  \Co$ with $\vert\supp(W)\cap\supp(X)\vert \leq 3$, $W\perp X$.}

\noindent {\em Proof.} The claim is either trivial or clear by definition if $\vert\supp(W)\cap\supp(X)\vert \leq 2$. So consider $X\in \Co$ with $\vert\supp(X)\cap\supp(W)\vert =3$, and by way of contradiction let $\supp(W)\cap\supp(X)=\{e,f,g\}$ so that $P_{X,W}$ is contained in a closed half-circle and includes a point in the interior of this half-circle.

By Lemma \ref{lemma_treel} applied to $\Mm^*$, there is a circuit $X'\in  \Co$ and two elements of $\{e,f,g\}$ (say, $e,f$) such that $\supp(X')$ and $\supp(X)$ are a modular pair in $\Mm$, and $e\in\supp(X')\cap\supp(W)\subseteq \supp(X)\cap\supp(W)\setminus f$, and since $\supp(X')\cap\supp(W)|\geq 2$, we know $\supp(X')\cap\supp(W)=\{e,g\}$. Multiplying by an element of $S^1$, we may assume $X'(e)=-X(e)$. Thus 
$$
\frac{X'(g)}{W(g)}=-\frac{X'(e)}{W(e)}=\frac{X(e)}{W(e)},
$$ 
and
$$
P_{X,W}=\bigg\{\frac{X'(g)}{W(g)},\frac{X(f)}{W(f)},\frac{X(g)}{W(g)}\bigg\}.
$$
In particular, these three points lie in the unit circle as described before.
Modular elimination of $e$ between $X'$ and $X$ gives a circuit $Y\in  \Co$ with $\supp(Y)\cap\supp(W)=\{f,g\}$, $Y(f)=X(f)$, and $Y(g)\in\pconv(\{ X(g), X'(g)\})$. Thus $P_{Y,W}$ lies in a half-open half-circle of $S^1$, contradicting $Y\perp W$.

{
\begin{figure}[h]
\begin{center}
%\scalebox{0.85}
{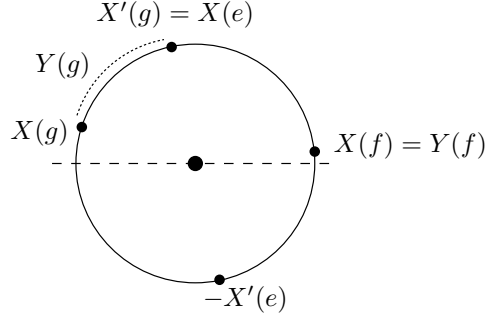}
\caption{Picture for the proof of Claim 2. 
%We omitted the division
  % by phases in $W$.
  }
\end{center}
\end{figure}}

\noindent{\bf Claim 3.} {\em $ \Do\perp \Co$.}

\noindent{\em Proof.} Induction on the rank of $\Mm$. If $\Mm$ has rank $2$, then all circuits have size $3$, and we conclude with Claim 2. Assume that $\Mm$ has rank $r>2$ and the claim holds for all matroids of rank $r-1$ or less.

Suppose by way of contradiction that there is $X\in  \Co$ and $W\in \Do$ with $X\not\perp W$. Choose $e\in E\setminus\supp(W)$. Then, $ \Co/e$ is a complex circuit orientation of the matroid $\Mm/e$ and $ \Do$ is a circuit signature of the matroid $\Mm^*\setminus e$ satisfying $X'\perp W'$ for all $X'\in \Co/e$ and  $W'\in  \Do\setminus e$ with $\vert \supp(X')\cap\supp(W')\vert \leq 2$. Since the rank of $\Mm/e$ is $r-1$, by induction hypothesis  $X'\perp W'$ for all $X'\in  \Co/e$, $W'\in  \Do\setminus e$.

Now look at our $X,W$ and choose $f\in \supp(X)\cap \supp(W)$. By the definition of contraction and deletion, $W\in  \Do\setminus e$ and there is $X'\in \Co/e $ with $X'\subseteq X$ and $f\in\supp(X')$. The vertices of $P_{X',W}$ are a subset of the vertices of $P_{X,W}$ - thus $X\not\perp W$ forces $X'\not\perp W$, contradicting the induction hypothesis.
\end{proof}

At last, we can justify Theorem~\ref{thm:elim} and Theorem~\ref{thm:dualpair:COM}.

\begin{cor}
The definition of complex matroids in terms of their oriented circuits obtained from axioms \textrm{($\Co 0$), ($\Co 1$), (ME)} is equivalent to the definition in terms of phirotopes (and, in turn, with the one in terms of dual pairs). 
\end{cor}
\begin{proof}
This is just a combination of Proposition \ref{phir_circ}, Proposition \ref{circ_perp} and Proposition \ref{perp-phir}.
\end{proof}

\subsection{Duality}

Given the set $\Co$ of phased circuits of a complex matroid, the corresponding set of phased cocircuits can be defined by orthogonality.

\begin{prop}\label{prop:circ:min} Let $\Co\subseteq (S^1\cup\{0\})^E$ be a complex circuit orientation of $\Mm$. 
Then the set of elements of $\Co^\perp\setminus\{\zero\}$ of minimal support is exactly the complex signature $\DD$ of $\Mm^*$ given by Proposition~\ref{circ_perp}. 
\end{prop}

\begin{proof}
Recall $$\Co^\perp=\{W\in (S^1\cup\{0\})^E\mid W\perp X \textrm{ for all }X\in\Co\}.$$

For any collection of phased sets, $\mathcal{T}$, let $\lfloor\mathcal{T}\rfloor$ denote
 the elements of $\mathcal{T}^\perp\setminus\{\zero\}$ with minimal support.

By Proposition \ref{circ_perp}, we have $\Do\subset\Co^{\perp}$. Since $\supp(\DD):=\{\supp(X)\mid X\in \DD\}$ is the set of circuits of the underlying matroid, by \cite[Proposition 2.1.20]{Oxley} it can be written as $\supp(\DD)=\lfloor\So\rfloor$,
where
$$\So:=\{A\subseteq E\mid \vert A\cap \supp(X) \vert\neq 1 \forall X\in \Co\}.$$ 
Now, $\supp(\Co^{\perp})\subset \So$ (since $X\perp W$ forbids $\vert\supp(X)\cap\supp(W)\vert =1$), and so 

\noindent (1) $\DD\subseteq \lfloor\Co^\perp\rfloor$ because for every $W\in\Co^{\perp}$ there is $Y\in\DD$ with $\supp(Y)\subseteq\supp(W)$,

\noindent (2) $\DD\supseteq\lfloor\Co^\perp\rfloor$ because every $W\in\lfloor\Co^\perp\rfloor$ has the same support as some $Y_W\in\DD$, and one sees as in the proof of Proposition \ref{circ_perp} that for any $X\in S^1\cup\{0\}$ with $\supp(W)\in\supp(\DD)$ the condition $W\perp\Co$ determines the ratios $W(f)/W(e)$ uniquely for every pair $e,f\in\supp(W)$. Thus, $Y_W=W$.
\end{proof}

\section{Two counterexamples}

\subsection{Strong elimination}\label{ex:strongelim}

This section gives the example promised in Remark~\ref{obs:modular}, demonstrating that our phased circuit axioms 
{\em cannot} include a general Elimination Axiom analogous to that for oriented matroids (Axiom $\Com 2$ in Definition \ref{OMaxioms}). 

\begin{example}\label{ex:el:bad} Let $v_1,\ldots,v_7$ denote the columns of the following matrix:

\[ M:=
\left(\begin{array}{rrrrrrr}
1  & 0 & -1 & 0 & 0 & i& 1-i \\
2  & -1& 0  &-1 & 0 &-i& 3+i \\
-i & 0 & -i & 0 & 2i&-i& -2i \\
-1  & 0 & 0  &-i &i+1& 0& -2  \\
\end{array}\right)\]

The vectors $(1,1,1,1,1,0,0)$ and $(-1,0,0,1,1,1,1)$ are both elements of $\ker(M)$ of minimal support, giving rise to two phased circuits $X:=(1,1,1,1,1,0,0)$ and $Y:=(-1,0,0,1,1,1,1)$. Now, a ``general" elimination axiom should describe the phases of the circuit obtained by eliminating $v_1$ from $X$ and $Y$ in terms of the phases of $X$ and $Y$.  This circuit should have support contained in $\{v_2,\ldots,v_7\}$, and below we list all circuits with such support (up to multiplication by a scalar).

{\small
\begin{align*}
-(1+i)v_2+ iv_3+v_4+\bigg(\frac{1}{2}+\frac{i}{2}\bigg)v_5+v_6&=0, \\
(2+i)v_2 + (1-i)v_3+v_4+ \bigg(\frac{3}{2}-\frac{1}{2}\bigg)v_5+v_7&=0,\\
-\frac{5i}{2}v_2+\bigg(-1+\frac{1}{2}\bigg)v_3 +v_4+  \bigg(1+\frac{i}{2}\bigg)v_6 -\frac{i}{2}v_7&=0,
\end{align*}
\begin{align*}
 \bigg(\frac{1}{2}+\frac{5}{2}i\bigg)v_2+ \bigg(\frac{3}{2}-\frac{i}{2}\bigg)v_3 +v_5  - \bigg(\frac{1}{2}+\frac{i}{2}\bigg)v_6+ \bigg(\frac{1}{2}+\frac{i}{2}\bigg)v_7 &=0,\\
 \bigg(\frac{3}{5}-\frac{4}{5}i\bigg)v_2 +v_4 + \bigg(\frac{7}{10}-\frac{i}{10}\bigg)v_5 + \bigg(\frac{3}{5} +\frac{i}{5}\bigg)v_6 + \bigg(\frac{2}{5}-\frac{i}{5}\bigg)v_7&=0,\\
\bigg(\frac{7}{13}+\frac{4}{13}i\bigg)v_3 + v_4+  \bigg(\frac{25}{26}+\frac{5}{26}i\bigg)v_5 + \bigg(\frac{8}{13}-\frac{1}{13}i\bigg)v_6 + \bigg(\frac{5}{13}+\frac{i}{13}\bigg)v_7&=0.\\
\end{align*}}
But we could construct a matrix with, for instance, all real entries and with $(1,1,1,1,1,0,0)$ and $(-1,0,0,1,1,1,1)$ in its kernel, and thus with $X$ and $Y$ in the resulting phased circuit set. A general elimination axiom should give the same elimination of $v_1$ from $X$ and $Y$ in both of these complex matroids, but of course it will not.
\end{example}

\subsection{Vectors}\label{Sect:vectors}\label{sec:vectors}

In this section we show that there is no "phased vector axiomatization" for oriented matroids analogous to the $\Vom$ axioms in Definition~\ref{OMaxioms}. Such an axiomatization should 
\begin{itemize}
\item be cryptomorphic to the other axiomatizations, and
\item have the property that, for complex subspaces $W$ of $\mathbb C^n$, the complex matroid of $W$ has vector set $\{\ph(v):v\in W\}$.
\end{itemize}
We give here an example to show that the circuits of a complex matroid with realization $W$ do not determine $\{\ph(v):v\in W\}$. 

%Sadly, there is no "phased vector axiomatization" for complex matroids  that is cryptomorphic to the other axiomatizations and has the property that, for complex subspaces $W$ of $\mathbb C^n$, the complex matroid of $W$ has vector set $\{\ph(v):v\in W\}$. In this section we give an example to show that the circuits of a complex matroid with realization $W$ do not determine $\{\ph(v):v\in W\}$.

Let $W_1$ be the row space of
$$\left(\begin{array}{cccc}
1&1+i&1&0\\
1+i&3i&0&1
\end{array}\right)$$
and let $W_2$ be the row space of
$$\left(\begin{array}{cccc}
1&1+i&1&0\\
1+i&4i&0&1
\end{array}\right).$$
We shall verify that $W_1$ and $W_2$ have the same complex matroid, but that there is a $v\in W_1$ such that $\ph(v)\neq\ph(w)$ for every $w\in W_2$.

For each $W_i$, the underlying matroid is uniform, of rank 2, with 4 elements, so has 4 (unphased) circuits. Thus each complex matroid has circuit set consisting of four $S^1$ orbits. We can read two of the orbits for each $W_i$ directly from the presentation above: each of the two complex matroids has $\ph(1,1+i,1,0)$ and $\ph(1+i,3i,0,1)=
\ph(1+i,4i,0,1)$ as circuits. To see the remaining two orbits, we perform Gauss-Jordan elimination on the two matrices:
$$\left(\begin{array}{cccc}
1&1+i&1&0\\
1+i&3i&0&1
\end{array}\right)
\rightarrow
\left(\begin{array}{cccc}
1&0&3&-1+i\\
0&1&i-1&-i
\end{array}\right)
$$
and $$\left(\begin{array}{cccc}
1&1+i&1&0\\
1+i&4i&0&1
\end{array}\right)
\rightarrow
\left(\begin{array}{cccc}
1&0&2&\frac{1}{2}(-1+i)\\
0&1&\frac{1}{2}(i-1)&\frac{-i}{2}
\end{array}\right).$$
So, the two $W_i$ give the same complex matroid.

On the other hand, note that $(2+i,1+4i,1,1)\in W_1$. Assume by way of contradiction that some $w\in W_2$ has 
$\ph(w)=\ph(2+i,1+4i,1,1)$. Then $w=k(1,1+i,1,0)+l(1+i, 4i, 0, 1)$ for some $k$ and $l$. To have the correct signs on the last two components, $k$ and $l$ must both be positive real numbers. However, one easily checks that no such $k$ and $l$ give the correct sign on the first two components.

\section{Weak maps and strong maps}\label{weakstrong}\label{sec:weakstrong}

Intuitively, a weak map of matroids is the combinatorial analog to moving a subspace of a vector space $\mathbb K^n$ into more 
special position with respect to the coordinate hyperplanes. The same intuition motivates the definition of weak maps for oriented matroids, although the intuition is known to be somewhat problematic in this case: there are weak maps of realizable oriented matroids which do not arise from geometrically ``close'' realizations (cf. Proposition 2.4.7 in~\cite{BLSWZ}).

A strong map of (oriented) matroids is the combinatorial analog to taking a subspace of a vector space. In the case of oriented matroids, this analogy has a beautifully straightforward interpretation via the Topological Representation Theorem. The covectors of a rank $r$ oriented matroid $\Mom$ label the cells in a regular cell decomposition of $S^{r-1}$, and the covectors of any rank $k$ strong map image of $\Mom$ label the cells in the intersection of this cell complex with a $(k-1)$-dimensional ``pseudoequator''. For details of this, see \cite[Section 7.7]{BLSWZ}. Thus strong maps of oriented matroids have a straightforward definition in terms of covectors (and hence also in terms of vectors), but it is not so clear how to see strong maps directly in terms of circuits, cocircuits, or chirotopes. As far as we know there is no definition of strong maps of oriented matroids in terms of circuits, cocircuits, or chirotopes without involving composition somehow. From the perspective of the Topological Representation Theorem, such a definition seems unlikely: the cocircuits of an oriented matroid represent only the vertices in the cell decomposition of $S^{r-1}$, and without referring to composition it's not clear how to describe how arbitrary pseudoequators intersect the entire cell decomposition. For the same reasons, it seems unlikely that we can define strong maps of complex matroids without vector axioms.

On the other hand, this section will develop a notion of weak maps of complex matroids that behaves much like weak maps of oriented matroids.

Recall the partial order on $(S^1\cup\{0\})^E$: we order $S^1\cup\{0\}$ to have unique minimum $0$ and all other elements maximal, and then order $(S^1\cup\{0\})^E$ componentwise. Also recall \cite[Proposition 7.3.11]{Oxley} that for matroids $\Mm_1$ and $\Mm_2$ on the same ground set $E$, there is a weak map from $\Mm_1$ to $\Mm_2$ if and only if every circuit in $\Mm_1$ contains a circuit of $\Mm_2$.

\begin{defn} Let $\Mo_1$ and $\Mo_2$ be complex matroids on the set $E$ with circuit sets $\Co_1$ resp. $\Co_2$. We say there is a {\em weak map} from $\Mo_1$ to $\Mo_2$, and write $\Mo_1\leadsto \Mo_2$, if for every $X\in \Co_1$ there exists $Y\in \Co_2$ such that $X\geq Y$.
\end{defn}

\begin{prop} \label{underlying weak map} Let $\Mo_1$ and $\Mo_2$ be complex matroids with underlying matroids $M_1$ resp. $M_2$.
\begin{enumerate} 
\item If $\Mo_1\leadsto \Mo_2$ then $M_1\leadsto M_2$.
\item If $\Mo_1\leadsto \Mo_2$ then $\rank(\Mo_1)\geq \rank(\Mo_2)$.
\end{enumerate}
\end{prop}

\begin{proof} The first statement is clear from the definition of weak maps, and the second statement follows from the first.
\end{proof}

\begin{prop}\label{prop:equivweak} Let $\Mo_1$ and $\Mo_2$ be complex matroids of the same rank and on the same ground set. Let $\phr_1$ and $\phr_2$ be phirotopes for $\Mo_1$ and $\Mo_2$, and let $\phr_1$ and $\phr_2$ be their duals. The following are equivalent.
\begin{enumerate}
\item $\Mo_1\leadsto\Mo_2$.
\item For some $c\in S^1$, $\phr_1\geq c\phr_2$.
\item For some $c\in S^1$, $\phr_1^*\geq c\phr_2^*$.
\end{enumerate}
\end{prop}

\begin{proof} The equivalence of the latter two statements is clear from Theorem~\ref{dualCOM}.  Let $\Mm_1$ and $\Mm_2$ denote the underlying matroids of $\Mo_1$ and $\Mo_2$, respectively.

If $\Mo_1\leadsto \Mo_2$ then by Lemma~\ref{underlying weak map}.1 we know that every basis of $\Mm_2$ is also a basis of $\Mm_1$. In particular, we have the following.
\begin{enumerate}
\item There exists $B_0$ an ordered basis of both $\Mm_1$ and $\Mm_2$. Without loss of generality assume $\phr_1(B_0)=\phr_2(B_0)$.
\item The basis graph of $\Mm_2$ is a subgraph of the basis graph of $\Mm_1$.
\end{enumerate}

\noindent For any ordered sequence $S$, let $\ul{S}$ denote the set of elements of $S$.

We will induct on distance from $\ul{B_0}$ in the basis graph of $\Mm_2$ to see that $\phr_1$ and $\phr_2$ coincide on every ordered basis $B$ of $\Mm_2$. If $\ul{B}\neq\ul{B_0}$, by basis exchange we can find $B_1$ a basis closer to $\ul{B_0}$ such that $\ul{B}=\{e, x_2, \ldots,  x_r\}$ and $B_1=\{f, x_2, \ldots,  x_r\}$ for some $e, f, x_2, \ldots,  x_r$. Then by Theorem~\ref{thm:elim}, any signature $X\in\Co_{\phr_1}$ on the basic circuit of $f$ with respect to $B$ satisfies 

\[\frac{X(e)}{X(f)}=-\frac{\phr_1(f, x_2, \ldots, x_r)}{\phr_1(e, x_2, \ldots, x_r)}=-\frac{\phr_2(f, x_2, \ldots, x_r)}{\phr_1(e, x_2, \ldots, x_r)}.\]

But $X\geq Y$ for some $Y\in\Co_{\phr_2}$, and $Y$ is a circuit signature in $\Mo_2$ on the basic circuit of $f$ with respect to $B$ in $\Mm_2$. So
\[
-\frac{\phr_2(f, x_2, \ldots, x_r)}{\phr_2(e, x_2, \ldots, x_r)}=\frac{Y(e)}{Y(f)}
=\frac{X(e)}{X(f)}
\]

\noindent and thus $\phr_2(f, x_2, \ldots, x_r)=\phr_1(f, x_2, \ldots, x_r)$.

Our proof that the second statement implies the first is adapted from \cite{BLSWZ} and is by induction on $|E|$.

Recall that a loop of a matroid is an element $e$ such that  $\{e\}$ is a circuit, and a coloop is an element  $e$ such that  $\{e\}$ is a cocircuit. Loops and coloops of complex matroids are loops or coloops of the underlying matroid. Write $\Co_1:=\Co_{\phr_1}$ and $\Co_2:=\Co_{\phr_2}$ for the sets of circuits of $\Mo_1$ and $\Mo_2$ respectively. First note:
\begin{itemize}
\item If $\Mo_1$ has a loop $e_0$, then $e_0$ is also a loop of $\Mo_2$, and the induction hypothesis tells us that $\Mo_1\setminus e_0 \leadsto\Mo_1\setminus e_0$, hence $\Mo_1\leadsto\Mo_2$.
\item If $\Mo_1$ has no loops but $\Mo_2$ has a coloop $e_0$, then $\phr_1/e_0\leadsto\phr_2/e_0$ and $\{e_0\}\not\in\CC_1$, so for every $X\in\CC_1$ there is a $Y\in\CC_2$ such that $X\setminus e_0\geq
 Y\setminus e_0
$. Since $e_0$ is a coloop, this implies $Y(e_0)=0$, so $X\geq Y$.
\end{itemize}

So consider the case when $\phr_1\geq\phr_2$ and $\Mo_2$ has no coloops. Let $X\in\CC_1$. Let $A$ be a maximal subset of $\supp(X)$ that's independent in $\Mo_2$, and extend $A$ to a basis $B$ of $\Mo_2$. Let $\tilde{\phr_1}$, $\tilde{\phr_2}$ be the restrictions of $\phr_1$ and $\phr_2$ to $(\supp(X)\cup B)^r$. Then $\tilde{\phr_1}\geq \tilde{\phr_2}$.

If $A:=E\setminus (\supp(X)\cup B)\neq \emptyset$ then, since $X\in\CC_1\setminus A$,
 the induction hypothesis tells us that there is a $Y\in\CC_2\setminus A \subseteq\CC_2$ such that $X\geq Y$.
 
 If $\supp(X)\cup B= E$, we can see that $B\subsetneq\supp(X)$. Otherwise, any
 $b\in B\setminus \supp(X)$ satisfies $\rank_{\Mo_2}(\supp(X)\cup(B\setminus b))<\rank_{\Mo_2}(\supp(X)\cup B)$. Thus $b$ is a coloop of $\Mo_2$, but $\Mo_2$ has no coloops. Thus $\supp(X)$ is a circuit of $\ul{\Mo_2}$. 

An easy induction on the rank shows that whenever $\Mm_1$ and $\Mm_2$ are matroids of the same rank such that every circuit of $\Mm_1$ is a circuit of $\Mm_2$, then $\Mm_1=\Mm_2$.

We conclude $\ul{\Mo_1}=\ul(\Mo_2)$, and so $\phr_1\geq\phr_2$ implies $\phr_1=\phr_2$. Thus $\Mo_1=\Mo_2$.
\end{proof}

As with realizable oriented matroids, weak maps of realizable complex matroids can arise from moving subspaces into more special position with respect to the coordinate hyperplanes. To make this precise, we give here the complex version of the same argument for oriented matroids (cf.~\cite{AnDa}). Consider the complex Grassmannian $G(r, {\mathbb C}^n)$, the topological space of all rank $r$ subspaces of ${\mathbb C}^n$. For any $W\in G(r, {\mathbb C}^n)$, let $\mu(W)$ be the corresponding rank $r$ complex matroid. Thus, if $W=\row(M)$, the function
 $\phr_M:[n]^r\to S^1\cup\{0\}$ taking
each $(e_1, \ldots, e_r)$ to the sign of the minor of $M$ with columns indexed by $(e_1, \ldots, e_r)$ is a phirotope for $\mu(W)$.

The following is our central result on the realizable interpretation of weak maps:
\begin{thm}\label{GC:thm} Let $\MM_1$ and $\MM_2$ be rank $r$ complex matroids on the ground set $[n]$. If $\overline{\mu^{-1}(\MM_1)}\cap\mu^{-1}(\MM_2)\neq \emptyset$ then $\MM_1\leadsto\MM_2$.
\end{thm}

\begin{proof} For any $r$-subset $B$ of $[n]$, let $U_B\subset G(r, {\mathbb C}^n)$ be the set of all row spaces of $r\times n$ complex matrices such that the square submatrix with column set indexed by $B$ is the identity. Then $U_B\cong{\mathbb C}^{r\times(n-r)}$, and the set of all $U_B$ is an atlas on $G(r, {\mathbb C}^n)$. Thus $U_B\cap \overline{\mu^{-1}(\MM_1)}\cap\mu^{-1}(\MM_2)\neq \emptyset$ for some $B$. Without loss of generality assume $B=[r]$. Thus we can (and will) identify $U_B$ with the set of $r\times n$ matrices $M$ of the form $(I|M')$, where $I$ is the $r\times r$ identity matrix, and $M'$ is a $r\times (n-r)$ matrix.

Now consider the two maps
$$U_B\stackrel{d}{\longrightarrow}{\mathbb C}^{[n]^r}\stackrel{\ph}{\longrightarrow}(S^1\cup\{0\})^{[n]^r}$$
where $d(M)(e_1,\ldots, e_r)$ is the $(e_1,\ldots, e_r)$ minor of $M$ (that is, the determinant of the submatrix of $M$ with columns indexed by $(e_1,\ldots, e_r)$, in that order). The composition of these two maps takes each $W$ to the phirotope for $\mu(W)$ with value 1 on $(1, 2, \ldots, r)$.

The map $d$ is continuous, hence the hypothesis gives \[ \overline{d(\mu^{-1}(\MM_1))}\cap d(\mu^{-1}(\MM_2))\neq \emptyset.\] But for each $i$, $d(\mu^{-1}(\MM_i))\subseteq\ph^{-1}(\phr_{M_i})$, so  $\overline{\ph^{-1}(\phr_{\MM_1})}\cap\ph^{-1}(\phr_{\MM_2})\neq \emptyset$. In particular, for every $X\in [n]^r$, we have \[\overline{\ph^{-1}(\phr_{\MM_1}(X))}\cap\ph^{-1}(\phr_{\MM_2}(X))\neq \emptyset.\] Notice that, for every $c\in S^1\cup\{0\}$, $\ph^{-1}(c)={\mathbb R}_+c$. Thus, for any $c_1,c_2\in  S^1\cup\{0\}$, \begin{center}$\overline{\ph^{-1}(c_1)}\cap \ph^{-1}(c_2)\neq\emptyset$ {\em if and only if} $c_1\geq c_2$.\end{center}
So $\phr_{\Mo_1}\geq \phr_{\Mo_2}$, and by Proposition~\ref{prop:equivweak} this means $\Mo_1\leadsto \Mo_2$.
\end{proof}

\setcounter{section}{0}

\renewcommand \thesection{{\sc Appendix} %\Roman{section}
}
\section{Matroids and oriented matroids}\label{appendix:matroids}

\renewcommand \thesection{A}

We give a quick introduction to matroids and oriented matroids by
stating the relevant definitions and results needed in 
the main body of the paper. We will omit proofs that can be found in
the literature. 
For matroid theory we follow the notation of~\cite{Oxley} and recommend this text for a reference; a similar text for oriented matroids is~\cite{BLSWZ}. 
In particular, we will follow the notational convention of writing $x$ for the singleton set $\{x\}$ whenever this will not cause confusion. 

We present the philosophy of oriented matroids as ``matroids with extra
structure''. This makes our presentation mildly unorthodox but leads naturally to our approach to complex
matroids. 
\subsection{Matroids}

\subsubsection{Matroid axioms}\label{matroidaxioms}

We start by presenting some well-known axiomatizations of matroids.

\begin{defn}[See Chapter 1 of \cite{Oxley}]\label{def:mat}$\,$
\begin{list}{\labelitemi}{\leftmargin=1.5em}
\item[1.]  A family $\Bm\subseteq 2^E$ of subsets of $E$ is the set of {\em bases of a matroid} $\Mm$ if and only if $\Bm\neq \emptyset$ and
\begin{enumerate}
\item[($\Bm 1$)] given $B_1,B_2\in\BB$ and $e\in B_1\setminus B_2$, there is $f\in B_2\setminus B_1$ such that $(B_1\setminus e)\cup f \in\Bm$ {\em (the Basis Exchange Axiom)}.
\end{enumerate}

\item[2.] A family $\Vm\subseteq 2^E$ is the set of {\em vectors}  of a matroid on the ground set $E$ if and only if $E\in\Vm$ and
\begin{enumerate}
\item[($\Vm 1$)] if $X_1, X_2\in\VV$ then $X_1\cup X_2\in\Vm$
\item[($\Vm 2$)] if $X\in\Vm$ and $\{Y_1, \ldots, Y_k\}$ is the set of maximal elements of $\Vm$ properly contained in $X$, then the sets $X-Y_1, \ldots, X-Y_k$ partition $X$.
\end{enumerate}

\item[3.] \cite[Definition... ]{Oxley}   A family $\Cm\subset 2^E$ is the set of {\em circuits} of a matroid on the ground set $E$ if and only if $\emptyset\not\in\Cm$ and
\begin{enumerate}
\item[($\Cm 1$)] if $C_1,C_2\in\Cm$ and $C_1\subseteq C_2$, then $C_1=C_2$ {\em (Incomparability)}. 

\item[($\Cm 2$)] if $C_1,C_2\in\Cm$ are distinct and there is an element $e\in E$ with $e\in C_1\cap C_2$, then there is $C_3\in\Cm$ with $C_3\subseteq (C_1\cup C_2)\setminus e${\em (Elimination)}.
\end{enumerate}
\end{list}

\end{defn}

\noindent To briefly state the cryptomorphisms: 
\begin{itemize}
\item Given  $\Bm$ the set of bases of a matroid, we say $A\subseteq E$ is {\em dependent} if it is not contained in a basis. The set $\Cm$ of all minimal dependent sets is the set of circuits of a matroid.
\item Given $\Cm$ the set of circuits of a matroid, $\Vm$ is the set $U_\Sm$ of all unions of elements of $\Cm$ (including the empty union).
\item Given $\Vm$ the set of vectors of a matroid, we say that $A\subseteq E$ is a {\em basis} if $A$ is maximal among sets not containing a vector. The set $\Bm$ of all bases is the set of bases of a matroid.
\end{itemize}

In this paper we need the following strengthening of the circuit axioms, recently proved in \cite{Delu09}.

\begin{defn}\label{modgen}\label{def:mat:mod:circ} Let $\Sm$ be a collection of incomparable nonempty subsets of a ground set $E$. Consider the poset obtained by partially ordering $U_{\Sm}:=\{\bigcup\Km\mid \Km\subseteq \Sm\}$ by inclusion. This poset is an atomic lattice, so the {\em meet} $A\vee B$ is defined for every pair $A,B$ of its elements (see \cite[Chapter 3]{Stanley}). 

Two elements $A$, $B$ of $\Sm$ give a {\em modular pair} if the longest chain from the minimal element $\emptyset$ of $U$ to their meet $A\vee B$ has length $2$. 
\end{defn}

\begin{lemma}[\cite{Delu09}]\label{lemma:modelim}
A collection $\Cm$ of incomparable nonempty subsets of a ground set $E$ is the set of circuits of a matroid if and only if the Elimination property $(\Cm 2)$ of Definition~\ref{def:mat} holds for all modular pairs $C_1,C_2$ of elements of $\Cm$.
\end{lemma}

\subsubsection{Duality and minors}

\begin{defn}\label{def:mat:orth} For $\Sm\subseteq 2^E$, we define
$\Sm^\perp:=\{A\subseteq E\mid \forall B\in \Sm \ |A\cap B|\neq 1\}$.
\end{defn}
\begin{thm}(cf.~\cite{Oxley})\label{dual:mat} If $\Mm$ is a matroid with ground set $E$, basis set $\Bm$, vector set $\Vm$, and circuit set $\Cm$, then there is a matroid $\Mm^*$ with ground set $E$, basis set $\Bm^*:=\{E\setminus X \mid X\in \Bm\}$, vector set $\Vm^*:=V^\perp$, and circuit set $\Cm^*$ the set of minimal nonempty elements of $V^*$.

If $\Mm$ is realized by a matrix with row space $W$, then $\Mm^*$ is realized by a matrix with row space $W^\perp$.
\end{thm}

\begin{defn} The matroid $\Mm^*$ in the statement of the previous theorem is called the {\em dual} to $\Mm$. 
The sets $\Vm^*$ and $\Cm^*$ of the previous theorem are called the set of {\em covectors} resp. {\em cocircuits} of $\Mm$.
\end{defn}

% We will make frequent use of the following basic fact. It follows immediately from our definitions, but we state it here for later reference.

\begin{lemma}[Proposition 2.1.20 of \cite{Oxley}]\label{two_el}
Let $C$ be a circuit and $D$ be a cocircuit of a matroid $\Mm$. Then $\vert C\cap D\vert\neq 1$. In fact, the set
$$\min\{D\subseteq E \mid D\neq \emptyset,\,\vert D\cap C\vert \neq 1 \textrm{ for all } C\in\Cm\},$$
where $\min$ denotes inclusion-minimality, is the set of cocircuits of $\Mm$.
\end{lemma}

\begin{defn}[Section 3.1 of \cite{Oxley}]\label{def:mat:minor:bases}
Let $\Mm$ be a matroid on the ground set $E$ with set of bases $\Bm$, and let $A\subseteq E$. Choose $\{a_{1},\ldots,a_l\}$ a maximal independent set in $A$. We define 
\begin{itemize}
\item[(1)] the {\em contraction} $\Mm / A$ as the matroid given by the set of bases $$\Bm(\Mm /A):=\{B\subset E\mid B\cup\{a_1,\ldots,a_l\}\in\Bm\}$$
\item[(2)] the {\em deletion} $\Mm\setminus A$ as the matroid with set of bases $$\Bm(\Mm \setminus A):=\max\{B\setminus A\mid B\in\Bm\},$$
where $\max$ denotes inclusion-maximality.
\end{itemize}
\end{defn}

For any $A\subseteq E$, we let $\Mm(A)$ denote $\Mm\backslash(E\backslash A)$. The matroids $\Mm/A$, $\Mm\setminus A$, $\Mm(A)$ are called {\em minors} of $\Mm$. In fact, in the representable case they encode data related to the minors of the original matrix.

\begin{lemma}[Section 3.1 of \cite{Oxley}]\label{def:mat:minor:circ}
The contraction and deletion of a matroid $\Mm$ on the ground set $E$ can also be defined by means of their set of circuits:
$$\Cm(\Mm\setminus A) = \{C\in\Cm(\Mm)\mid C\cap A =\emptyset\},$$
$$\Cm(\Mm / A)=\min\{C\setminus A \mid C\in\Cm(\Mm),\,C\not\subseteq A\}.$$
Moreover, the operations of contraction and deletion are dual to each other in the sense that
$$(\Mm / A)^*=\Mm^*\setminus A.$$
\end{lemma}

\subsubsection{Rank, closure, flats}

%\subsubsection{Rank}

It is easy to check from the definition that all bases of a matroid have the same size (\cite[Lemma 1.2.1]{Oxley}). Thus we can define the {\em rank} of a matrix to be the size of any basis.
 
\begin{defn}[Rank] Let $\Mm$ be a matroid on the ground set $E$ with set of bases $\Bm$, and let $A\subseteq E$. Define the {\em rank} of $A$ to be 
\[
\rank(A):=\max\{\vert A\cap B\vert\mid B\in \Bm\}.
\]
\end{defn}
\noindent Thus $\rank(\Mm)=\rank(E)$.

The notion of rank defines a closure operator on $E$:

\begin{defn}[Closure]
Let $\Mm$ be a matroid on the ground set $E$. Given $A\subset E$ define
\[
\cl(A):=\max \{A'\subseteq E\mid A\subseteq A', \rank(A)=\rank(A')\}.
\]
The function $\cl:E\to E $ is the {\em closure operator} of $M$.
\end{defn}

%\subsubsection{Flats} 
A {\em flat} of a matroid $\Mm$ is the complement of a vector of $\Mm$.  It is easy to see that flats can be defined as the subsets $A\subseteq E$ such that $\cl(A)=A$.

For a matroid $M$, we are interested in several posets, each ordered by inclusion:
\begin{itemize}
\item the poset of flats,
\item the poset of vectors ($U_\Cm$), and
\item the poset of covectors.
\end{itemize}
Each of these is a ranked lattice, with meet given by intersection and join given by union \cite{Oxley}. 

%An important property of the set of flats of a matroid is that it is closed under intersection and union. Indeed, the poset of flats,  partially ordered by inclusion, is a geometric lattice, with meet given by intersection and join given by union. (This gives rise to yet another cryptomorphic characterization of matroids \cite[Theorem 1.7.5]{Oxley}). It is easy to see that flats can be defined as the subsets $A\subseteq E$ such that $\cl(A)=A$.
%$\rank(A\cup e)>\rank(A)$ whenever $e\in E\setminus A$. 

\subsubsection{Modularity of circuit pairs}

The maximal proper flats of a matroid $\Mm$ are called {\em hyperplanes}.   Thus, in a matroid of rank $r$, all hyperplanes have rank $r-1$. One can check that the cocircuits of $\Mm$ are exactly the complements of hyperplanes of $\Mm$. 

\begin{defn}\label{modmat} Two elements $A$, $B$ of a ranked lattice $L$ are a  {\em modular pair} if \begin{center} $\rank(A)+\rank(B)=\rank(A\wedge B)+\rank(A\vee B)$.
\end{center}
\end{defn}

\begin{remark}\label{pofferbacco}
The above definition extends Definition \ref{modgen}.
In particular, we note
that two hyperplanes are a modular pair of flats if and only if their complements are a modular pair of cocircuits.
\end{remark}

\begin{lemma}\label{trans_mod}
Let $\Mm$ be a matroid on the ground set $E$, and let $e\in E$ be a nonloop. Then
\begin{itemize}
\item[(1)] if $C_1,C_2$ is a modular pair of circuits of $\Mm\setminus e$ then it is a modular pair of circuits of $\Mm$,
\item[(2)] if $C_1,C_2$ is a modular pair of circuits of $\Mm/e$ then
  the (unique) pair $C'_1$, $C'_2$ of circuits of $\Mm$ with $C'_1\subseteq C_1\cup e, C_2'\subseteq C_2\cup e$
   is modular. %a modular pair of circuits of $\Mm$.
\end{itemize}
\end{lemma}

\def\rd{r^*}
\def\rdc{\rd_{/e}}
\def\rdd{\rd_{\setminus e}}
\begin{proof} Claim (2) is proved as Lemma 2.3 in \cite{LasV}. We give here for completeness a proof of both claims. 

In view of Definition~\ref{def:mat:mod:circ} we show the equivalent statements about the dual $\Mm^*$. In what follows, $\rd$, $\rdd$, $\rdc$ are the rank functions of $\Mm^*$, $\Mm^*\setminus e$, $\Mm^*/e$ respectively.

\noindent (1)  Let $H_1,H_2$ be a modular pair of hyperplanes of $\Mm^*/e$. Then for $i=1,2$, $H_i':=H_i\cup e$ is a hyperplane of $\Mm^*$,
$$\rdc(H_1\cap H_2)=\rd((H_1\cap H_2) \cup e) - \rd(e)=\rd(H'_1\cap H_2')-\rd(e)$$ 
and since by assumption $\rdc(H_1\cap H_2)=\rdc(E\setminus e) -2 =\rd(E)-\rd(e)-2$, we have $\rd(H_1'\cap H_2')=\rd(E)-2$. So $H_1,H_2$ is a modular pair.

\noindent (2) Let $H_1,H_2$ be a modular pair of hyperplanes of $\Mm^*\setminus e$. For $i=1,2$ let $H_i'\subset H_i\cup e$ denote the hyperplane of $\Mm^*$ containing $H_i$. If $\rdd(E\setminus e)= \rd(E)$, then 
$$\rdd(E\setminus e) - 2=\rdd(H_1\cap H_2) =\rd(H_1\cap H_2) \leq \rd(H_1'\cap H_2') \leq \rd(E) -2$$

\noindent and $H_1', H_2'$ are a modular pair. If however $\rdd(E\setminus e) <\rd (E)$, then $e$ is in every basis of $\Mm^*$, and $e\in H_1'\cap H'_2$. Then 

$$\rd(E)-3=\rdd(E\setminus e)-2=\rd(H_1\cap H_2)=\rd(H_1'\cap H_2')-1$$
and $H_1',H_2'$ are a modular pair in $\Mm^*$.
\end{proof}

\subsubsection{The basis graph} We introduce a key tool in the proof of the cryptomorphism between the axioms for dual pairs and the phirotope axioms in Section~\ref{pairs>phiro}. 

\begin{lemma}[\cite{Oxley}, Corollary 1.2.6]\label{basic_circ}
If $B$ is a basis of a matroid $\Mm$ on the ground set $E$ and $e\in E\setminus B$ then there is a unique circuit $X\subseteq B\cup\{e\}$, called the {\em basic circuit} of $e$ with respect to $B$ and denoted by $C(B,e)$. In particular, for any pair of bases of the form $B_1=A\cup e_1, B_2=A\cup e_2$ there is a unique circuit supported on $B_1\cup B_2$.
\end{lemma}

\begin{defn} (\cite{Maurer}) \label{def:basisgraph}
 The {\em basis graph} of a matroid $\Mm$ with set of bases $\Bm$ is the simple graph with vertex set \[V(G):=\Bm\] and edge set \[E(G):=\{\{B_1,B_2\} \mid B_1=A\cup e_1, B_2=A\cup e_2 \mbox{ for some $e_1\neq e_2\in B_2\setminus A$} \}.\] 
\end{defn}

Thus the edge between two vertices $B_1=A\cup e_1$, $B_2=A\cup e_2$ can be associated with the circuit $C(B_1,e_2)=C(B_2,e_1)$. 

Maurer gave a thorough treatment of these graphs, giving for instance a complete characterization of which graphs are basis graphs of a matroid. For this paper we will only need the following Theorem~\ref{theo:maurer}. 

A sequence of edges $e_1,\ldots ,e_k$ in a graph $G$ is a {\em path from the vertex $A$ to the vertex $B$} if  for all $j\in\{1,\ldots , k-1\}$, $e_j$ and $e_{j+1}$ share a vertex  and if $A$ (resp.\ $B$) is the vertex of $e_1$ (resp.\ $e_k$) that is not shared with $e_2$ (resp.\ $e_{k-1}$). We say that an {\em elementary move} on the given path is the substitution of any subpath $e_je_{j+1}$ with another path consisting of at most two edges of $G$, and such that the replacement yields again a path. The {\em trivial path} is the path corresponding to an empty sequence of edges.

\begin{thm}[\cite{Maurer}]\label{theo:maurer}
Let $G$ be the basis graph of a matroid $\Mm$ and choose a vertex $A$ of $G$. Then every closed path in $G$ from $A$ to $A$ can be reduced to the trivial path by a sequence of elementary moves and of inverses thereof.
\end{thm}

\subsection{Oriented matroids}
\subsubsection{Signs}

\begin{defn}\label{OM:sign:order} Given a finite ground set $E$, a {\em sign vector} (or {\em signed set}) is any $$X\in (S^0\cup\{0\})^E$$
where $S^0=\{+1, -1\}$ is the unit sphere in $\mathbb R$. We will denote by $X(e)$ the $e$-th component of $X$. The signed set with value $0$ for all components will be denoted by $\zero$.

We order $S^0\cup\{0\}$ according to the following Hasse diagram.
\begin{center}
\includegraphics[scale=0.4]{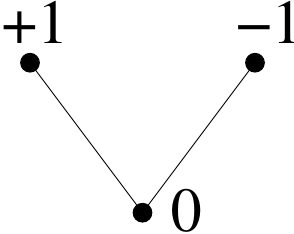}
\end{center}

The {\em sign} $\sign(x)$ of $x\in\mathbb R$ is defined to be 0 if $x=0$ and $\frac{x}{|x|}$ otherwise. The sign $\sign(v)$ of $v\in {\mathbb R}^E$ is defined to be the sign vector with $e$-th component $\sign(v_e)$. 
\end{defn}

\begin{defn} If $X,Y\in (S^0\cup\{0\})^E$, the {\em composition} $X\circ Y\in (S^0\cup\{0\})^E$ is defined as follows: for every $e\in E$, 
$$X\circ Y(e):=\left\{\begin{array}{ll} 
X(e)&\mbox{ if $X(e)\neq 0$}\\ 
Y(e)&\mbox{ otherwise.}\end{array}\right.$$
\end{defn}

We define the {\em convex hull} of a subset $P$ of $S^0\cup\{0\}$ to be the set of all signs of positive linear combinations of the elements of $P$. Thus
\begin{itemize}
\item $\conv(\emptyset)=\emptyset$
\item $\conv(\{0\})=\{0\}$
\item $\conv(\{0,\epsilon\})=\epsilon$ for $\epsilon\in S^0$
\item $\conv(P)=(S^0\cup\{0\})$ if $S^0\subseteq P$.
\end{itemize}

\subsubsection{Oriented matroid axioms}%\label{OMaxioms}
\begin{defn}\label{OMaxioms}$\,$
\begin{list}{\labelitemi}{\leftmargin=1.5em}
\item[1.] A function $E^d\to S^0\cup\{0\}$ is called a {\em rank d chirotope} of an oriented matroid $M$ if
\begin{itemize}
\item [($\chi\,$1)] $\chi$ is nonzero
\item[($\chi\,$2)] $\chi$ is alternating
\item[($\chi\,$3)] For any two subsets $x_1,\ldots,x_{d+1}$ and
  $y_1,\ldots , y_{d-1}$ of $E$, $0$ is contained in the convex hull
  of the numbers $$(-1)^k\chi(x_1,x_2,\ldots,\widehat{x_k},\ldots,
  x_{d+1})\chi(x_k,y_1,\ldots,y_{d-1}) $$ {\em (Combinatorial Grassmann-Pl\"ucker relations)}.

\end{itemize}

\item[2.] A family $\Vom\subseteq (S^0\cup\{0\})^E$ of signed sets is the set of {\em signed vectors} of an  oriented matroid $\Mom$ if
\begin{itemize}
\item[($\Vom 0$)] $\zero\in\Vom^*$,
\item[($\Vom 1$)] $\Vom=-\Vom$ {\em (Symmetry)},
\item[($\Vom 2$)] if $X,Y\in\Vom^*$ then $X\circ Y\in\Vom^*$ {\em (Composition)},
\item[($\Vom 3$)] for every $X,Y\in\Vom$ and $e\in E$  with $X(e)=-Y(e)$ there is some $Z\in\Vom$ with 
  \begin{itemize}
    \item[$\bullet$] $Z(f)\in\conv(\{X(f), Y(f)\})$ for all $f$, and 
    \item[$\bullet$] $Z(e)=0$
   \end{itemize} {\em (Vector Elimination)}.
\end{itemize}

\item[3.] A family $\Com\subseteq (S^0\cup\{0\})^E\setminus \{\zero\}$ of signed sets is the set of {\em signed circuits} of an  oriented matroid $\Mom$ if
\begin{itemize}
%\item[($\Com 0$)] $\zero\not\in\Com$
\item[($\Com 0$)] $\Com=-\Com$ {\em (Symmetry)},
\item[($\Com 1$)] if $X,Y\in\Com$ and $\supp(X)\subseteq\supp(Y)$ then $X=\pm Y$ {\em (Incomparability)},
\item[($\Com 2$)] for every $X,Y\in\Com$  and $e,f\in E$ with $X(e)=-Y(e)$ and $X(f)\neq-Y(f)$, there is some $Z\in\Com$ with 
\begin{itemize}
    \item[$\bullet$] $Z(g)\leq\max\{X(g), Y(g)\}$ for all $g$ for which this maximum exists, 
    \item [$\bullet$] $f\in\supp(Z)$, and
    \item[$\bullet$] $Z(e)=0$
   \end{itemize} 
 {\em (Circuit Elimination)}.
\end{itemize}
\end{list}
\end{defn}

As for matroids, there are cryptomorphisms allowing us to speak of ``the oriented matroid with chirotopes $\chi$ and $-\chi$, vector set $\Vom$, and circuit set $\Com$''. The {\em underlying matroid} of this oriented matroid has basis set $\supp(\chi)$, vector set $\{\supp(X): X\in\Vom\}$, and circuit set 
$\{\supp(X): X\in\Com\}$. 
%We define the {\em rank} of an oriented matroid to be the rank of its chirotope or, equivalently, the rank of the underlying matroid. 

\subsubsection{Modularity of signed circuit pairs}\label{appendix:OMmod}

The hypothesis in the Circuit Elimination Axiom $\Com 2$ in Definition \ref{OMaxioms} can be weakened to consider only {\em modular} circuits. The resulting Modular Elimination is equivalent to the full Circuit Elimination Axiom in the context of oriented matroids. In the setting of complex matroids, Modular Elimination is an axiom, while the full Circuit Elimination does not hold. (A counterexample is given in Section \ref{ex:strongelim}.)

\begin{prop}[Modular Elimination Axiom \cite{LasV}]\label{OMmodelim} In
  the definition of signed circuits (Definition \ref{OMaxioms}.3), % if
  % the set $\{\supp(C)\mid C\in \Com\}$ is known to be the set of
  % circuits of a matroid $M$ then 
the Circuit Elimination Axiom can be replaced by the {\em Modular Elimination Axiom:}
\begin{itemize}
\item [($\Com 2'$)] for every $X,Y\in\Com$ and $e,f\in E$ such that
\begin{itemize}
\item  [$\bullet$] $\supp(X)$, $\supp(Y)$ is a modular pair in $\{\supp(Z)\mid Z\in\Com\}$, 
\item [$\bullet$] $X(e)=-Y(e)$, and 
\item [$\bullet$] $X(f)\neq-Y(f)$, 
\end{itemize} 
there is some $Z\in\Com$ with
$f\in\supp(Z)$, $e\not\in\supp(Z)$, and for all $g$, $Z(g)\in\{0,X(g), Y(g)\}$.
\end{itemize}
\end{prop}

\begin{remark}
Proposition \ref{OMmodelim} does not assume that an underlying matroid
is already given, and thus it is a nontrivial strengthening of the
axiomatization for oriented matroids called `modular elimination' in
the book \cite{BLSWZ} and due to Las Vergnas \cite{LasV}.
\end{remark}

\subsubsection{Orthogonality}

\begin{defn} Two sign vectors $X,Y\in\{+,0,-\}^E$ are defined to be {\em orthogonal} if $0\in\conv(\{X(e)Y(e)\mid e\in E\})$. 
\end{defn}
 This definition is inspired by orthogonality of vectors in $\mathbb R^n$, and it leads to a definition of orthogonality of oriented matroids that nicely models  orthogonality of real vector spaces.

\begin{thm}\label{dualOM} If $\Mom$ is an oriented matroid with ordered ground set $E$, chirotope $\chi: E^r\to S^0\cup\{0\}$, circuit set $\Com$,
and vector set $\Vom$, then there is an oriented matroid $\Mom^*$  with 
\begin{itemize}
\item[(0)] ground set $E$, 
\item[(1)] chirotope $\chi^*: E^{|E|-r}\to\{0,+,-\}$ given by
$$\chi^*(x_1, \ldots, x_{n-r})=\chi(y_1, \ldots, y_r)\sigma(x_1, \ldots, x_{n-r},y_1, \ldots, y_r),$$
where $\{y_1, \ldots, y_r\}=E\setminus \{x_1, \ldots, x_{n-r}\}$ and $\sigma$ denotes the sign of the indicates permutation of $E$,
\item[(2)] covector set $\Vom^*=\Vom^\perp$, and
\item[(3)] cocircuit set $\Com^*=\min(\Vom^\perp-\{\zero\})$, \\where $\min$ denotes support minimality.
\end{itemize}

The underlying matroid of $\Mom^*$ is the dual of the underlying matroid of $\Mom$. If $\Mom$ is realized by a matrix with row space $W$, then $\Mom^*$ is realized by a matrix with row space $W^\perp$.
\end{thm}

This $\Mom^*$ is called the {\em dual} to $\Mom$. The vectors of $\Mom^*$ are the {\em covectors} of $\Mom$, and the circuits of $\Mom^*$ are called the {\em cocircuits} of $\Mom$.

\bibliographystyle{abbrv}
\bibliography{biblioCM}

\end{document}